\documentclass{article}

\usepackage[utf8]{inputenc} % allow utf-8 input
\usepackage[T1]{fontenc}    % use 8-bit T1 fonts
\usepackage{hyperref}       % hyperlinks
\usepackage{url}            % simple URL typesetting
\usepackage{booktabs}       % professional-quality tables
\usepackage{amsfonts}       % blackboard math symbols
\usepackage{nicefrac}       % compact symbols for 1/2, etc.
\usepackage{microtype}      % microtypography
\usepackage{lipsum}

\usepackage{graphicx}
\usepackage{amsmath}
\usepackage{amssymb}
\usepackage{amsthm}
\usepackage{amscd}
\usepackage{alltt}
\usepackage[all]{xy}

\newtheorem{lemma}{Lemma}
\newtheorem{prop}{Proposition}
\newtheorem{theorem}{Theorem}
\newtheorem*{theorem*}{Theorem}

\newcommand{\rar}{\rightarrow}

\newcommand{\dotminus}{\mathbin{\dot{-}}}
\newcommand{\dottimes}{\mathbin{\dot{\times}}}

\DeclareMathOperator\Mat{M}
\DeclareMathOperator\gl{GL}
\DeclareMathOperator\unit{U}

\DeclareMathOperator\eunit{EU}
\DeclareMathOperator\gunit{GU}

\DeclareMathOperator\Cent{C}
\DeclareMathOperator\diag{D}

\DeclareMathOperator\tr{tr}

\DeclareMathOperator\Norm{\mathrm{N}}
\DeclareMathOperator\Ker{Ker}

\renewcommand{\leq}{\leqslant}
\renewcommand{\geq}{\geqslant}
\renewcommand{\phi}{\varphi}

\newcommand{\eps}{\varepsilon}
\newcommand{\leqt}{\mathbin{\lefteqn{\text{\raisebox{1pt}{\(\lhd\)}}}\mspace{-5mu} \leq}}
\newcommand{\inv}[1]{\!\;\overline{\!\!\:#1\vphantom d\!\!\:}\;\!}

\DeclareMathOperator{\id}{id}
\DeclareMathOperator\herm{H}

\newcommand{\veccol}[2]{\bigl(\begin{smallmatrix} #1\\ #2 \end{smallmatrix}\bigr)}
\DeclareMathOperator{\Hom}{Hom}
\DeclareMathOperator{\End}{End}
\DeclareMathOperator{\Aut}{Aut}
\DeclareMathOperator{\Heis}{Heis}

\DeclareMathOperator{\met}{M}
\DeclareMathOperator{\hyp}{H}

\newcommand{\sub}[2]{{_{#1\!}{#2}}}

\makeatletter
\newcommand{\bigperp}{\mathop{\mathpalette\bigp@rp\relax}\displaylimits}
\newcommand{\bigp@rp}[2]{\vcenter{\m@th\hbox{\scalebox{\ifx#1\displaystyle2.1\else1.5\fi}{\(#1\perp\)}}}}
\makeatother

\newcommand{\bim} {\mathbf{Bim}}
\newcommand{\ibim}{\mathbf{iBim}}
\newcommand{\hbim}{\mathbf{hBim}}
\newcommand{\qbim}{\mathbf{qBim}}

\title{Groups normalized by the odd unitary group}

\author{
  Egor Voronetsky \\
  {\em \small Chebyshev Laboratory,} \\
  {\em \small St. Petersburg State University,} \\
  {\em \small 14th Line V.O., 29B,} \\
  {\em \small Saint Petersburg 199178 Russia} \\
  {\small voronetckiiegor@yandex.ru}
}

\begin{document}
\maketitle

\begin{abstract}
We will give a definition of quadratic forms on bimodules and prove the sandwich classification theorem for subgroups of the general linear group \(\gl(P)\) normalized by the elementary unitary group \(\eunit(P)\) if \(P\) is a nondegenerate bimodule with large enough hyperbolic part.
\end{abstract}

\section{Introduction}

The classical groups over fields, i.\,e. linear, orthogonal and symplectic groups (and their twisted forms, including all infinite families of simple Lie groups) were studied for quite a long time. In Emil Artin's book \cite{Artin} they are constructed from the point of view of projective geometry. Many properties of these groups over arbitrary division ring, such as normal subgroups, automorphisms and presentations, are written in Dieudonn\'e's papers \cite{Dieudonne2,Dieudonne1,Dieudonne3}.

Later the definitions of classical groups were generalized to the case of arbitrary underlying commutative rings. The most complicated case of orthogonal groups can be found in famous book \cite{Knus}. Many authors proposed a general definition of unitary groups that should include all classical groups considered before and that should work for arbitrary non-commutative rings (with involutions). The most important papers of this type are \cite{Tits,Wall,Weil}.

Nevertheless, all definitions of generalized unitary groups were unsatisfactory (first of all, they were different for rings with invertible \(2\) and for rings of characteristic \(2\)). In his dissertation \cite{BakThesis} Antony Bak gave the definition of form parameters for a non-commutative ring with an involution and used this to define the unitary group as the family of all module automorphisms that preserve a hermitian form and a quadratic form (defined through a form parameter) simultaneously. The definition allowed to transfer all results of the classical theory over division rings to the general case of even groups (this means the families \(A_l\), \(C_l\), and \(D_l\) for reductive group schemes) for those problems that have any sense in such generality. All these results were published in H. Bass's paper \cite{Bass1}.

Later Victor Petrov proposed a generalization of form parameters and corresponding groups to the odd case in paper \cite{OddPetrov}. Odd form parameters introduced there allows to give the right definition of unitary groups in full generality and to transfer many results about even groups to them. Unfortunately, odd form parameters in his definition depend on a module itself. This does not allow to construct the category of all such modules and, for example, to define an orthogonal sum operation on the category.

First several sections of the current paper, i.\,e. from \(2\) to \(6\), further generalize these ideas. For a give hermitian form we will define corresponding quadratic forms with values in special algebraic objects called quadratic structures (which generalize factors by a form parameter in Bak's theory or factor-groups of the Heisenberg groups by an odd form parameter in Petrov's theory). Moreover, our quadratic forms will be constructed not for right modules, but for bimodules, as in books \cite{Bak,Knus} in the case of hermitian forms. All triples consisting of a bimodule, a hemitian form and a corresponding quadratic form constitute a bicategory (with respect to tensor products) satisfying natural properties.

One of the most important problems about unitary groups is the classification of their normal subgroups (it is so-called sandwich classification because of the type of results). In the case of groups over division rings and non-degenerate forms these results are classical, including the simplicity theorems of the corresponding projective groups. In the case of general rings the problem is closely related with lower unitary \(K\)-theory (that is, generalizing of algebraic \(K\)-functors from linear to unitary groups)

A lot of papers deal with the sandwich classification of normal subgroups of classical groups (i.\,e. symplectic and orthogonal groups) and their generalizations. We pick out only \cite{Wilson,Golub,Li1,Li2,Vaser3,Vaser4,VaserYou}. Later these results were generalized for unitary Bak's groups in \cite{BakThesis,UnitaryII} and even for odd Petrov's groups in \cite{OddPreusser}. Also Raimund Preusser gave new short proofs of the sandwich classification of normal subgroups in the case of commutative rings in his recent papers \cite{EvenCommutative,OddCommutative}.

We should also mention the construction and the proof of normality of the elementary unitary group, since this is a basis for sandwich classification theorems. For the linear groups over commutative rings this is a famous Suslin's normality theorem \cite{Sus} and in the case of unitary groups there are separate papers \cite{BakVav,UnitaryI} concerning this problem.

Lower unitary \(K\)-theory is a theme of many papers. Let us mention only \cite{Bass1, Bass2, Hazrat, Tang1}, overview \cite{Bak65}, and book \cite{HahnO}. Similar results for classical groups were published in \cite{Vaser1,Vaser2}.

Another problem of big interest is the sandwich classification of intermediate subgroups between various classical groups, for example, between \(\mathrm{Sp}(2l, K)\) and \(\gl(2l, K)\). If \(K\) is a field, then the problem leads to the description of many maximal subgroups of classical groups. The sandwich classification theorems for overgroups of classical groups in the corresponding general linear group were proved for an arbitrary commutative ring \(K\) in papers \cite{OrthogonalEven,Symplectic,OrthogonalOdd,Hong1,Hong2}. Nowadays similar problems for various pairs of group schemes are a popular theme of research in St. Petersburg mathematical school, this is written in overviews \cite{Vav1} and \cite{VavStep}. In the case of Bak's unitary groups the sandwich classification for overgroups was obtained in Petrov's paper \cite{Petrov}.

The mail goal of this paper is to generalize all these results for odd unitary groups over a bimodule \(P\) with hermitian and quadratic forms. More precisely, we a going to classify all subgroups of the general linear groups \(\gl(P)\) normalized by the elementary unitary group \(\eunit(P)\). The exact statement is given below, all necessary definitions will given in the main text.

\begin{theorem*}
Suppose that \(K\) is a commutative ring with an involution, \(S\) and \(R\) are quadratic \(K\)-algebras, and \(P\) is a quadratic \(S\)-\(R\)-bimodule with non-degenerate hermitian form. Suppose also that \(\End(P)\) is a quasi-finite \(K\)-algebra, \(P = P_0 \perp \hyp(P_1) \perp \ldots \hyp(P_l)\), \(l \geq 4\), \(\End(P_0)\) is generated by less than \(4l^2\) elements as a \(K\)-module, and the bimodule \(P_i\) is isomorphic to a direct summand of \(P_j^N\) for \(N\) big enough and for any \(i, j \in \{-l, \ldots, -1, 1, \ldots, l\}\) (where \(P_{-i} = \inv{P_i^\vee}\) for \(i > 0\)). Then there exists explicitly described levels \(L\) and groups \(\eunit(P, L), \gunit(P, \lfloor L \rfloor) \leq \gl(P)\) such that every subgroup \(G \leq \gl(P)\) normalized by \(\eunit(P)\) satisfies the inequality \(\eunit(P, L) \leq G \leq \gunit(P, \lfloor L \rfloor)\) for unique level \(L\). Conversely, every group satisfying such an inequality is normalized by \(\eunit(P)\).
\end{theorem*}

In sections \(7\)--\(9\) we will give the definitions of elementary unitary groups, levels and level subgroups. We will also prove their simplest properties. These lemmas are interested not only because of the main theorem, but also since they may help to prove other results about unitary groups, such as lower \(K\)-theory.

The proof of the sandwich classification theorem occupies sections \(10\)--\(12\). The main techniques are the absolute Noetherian reduction and the extraction of transvections taken and modified from paper \cite{OddPreusser}. The sufficient information about inductive limits used (i.\,e. about quasi-finite algebras over commutative rings) can be found in \cite{NonabelianBak}. Further development of these ideas is in \cite{BakHazratVav}, including the localization-completion method applied to unitary \(K\)-theory.

In the final section we will derive several known classification theorems. These theorems include the sandwich classification of overgoups of symplectic and split orthogonal groups in the general linear group, and also the classification of normal subgroups of split orthogonal and symplectic groups.

The author wants to express his gratitude to his advisor Nikolai Vavilov and also to Antony Bak for ideas that lead to the definition of quadratic forms on bimodules.

The research is supported by «Native towns», a social investment program of PJSC «Gazprom Neft».

\section{General notions}

We will assume that all ring to be considered are associative with unity. For an arbitrary ring \(R\) let \(R^*\) be the group of invertible elements in \(R\), \(R^\bullet\) --- the multiplicative monoid of \(R\) (i.\,e. the set \(R\) with multiplication and unity), \(\Cent(R)\) --- the center of \(R\). The center of a group \(G\) will also be denoted as \(\Cent(G)\). If \(H \leq G\) a subgroup, then \(\Norm_G(H)\) and \(\Cent_G(H)\) are the corresponding normalizer and centralizer. For a commutative ring \(K\) with involution \(k \mapsto \inv k\) let \(\herm(K) = \{k \in K \mid k = \inv k\}\) be the subring of symmetric elements.

Recall the standard group theoritic identities \([xy, z] = {^x[}y, z] \, [x, z]\) and \([x, yz] = [x, y] \, {^y[}x, z]\), which will be used further. If \(G\) is a group and \(F, H \leq G\) are subgroups, then \([F, H]\) and \({^F H}\) are the subgroups of \(G\) generated by elements \([f, h]\) and \({^fh}\) for all \(f \in F\) and \(h \in H\). Clearly, \({^FH} = [F, H] H\).

We will also use the language of bicategories. Let us sketch one of possible definitions of these objects. A bicategory \(B\) is a set of its objects \(\mathrm{Ob}(\mathcal B)\) such that for all \(X, Y \in \mathrm{Ob}(\mathcal B)\) a category \(\mathcal B(X, Y)\) is given. These categories are called categories of morphisms from \(X\) to \(Y\), morphisms in them are called \(2\)-morphisms. Moreover, there are composition polyfunctors \(\otimes \colon \mathcal B(X_{n - 1}, X_n) \times \ldots \times \mathcal B(X_0, X_1) \rar \mathcal B(X_0, X_n)\) in order to multiply morphisms for all \(n \geq 0\) and \(X_0, \ldots, X_n \in \mathrm{Ob}(\mathcal B)\). Finally, these polyfunctors are coherent up to certain natural isomorphisms that are also coherent through some identitites. The \(0\)-ary compositions (i.\,e. the identity morphisms) are also denoted by \(I_X \in \mathcal B(X, X)\).

Let \(\bim_K\) be the bicategory of \(K\)-algebras and bimodules, where \(K\) is a commutative ring. Objects in \(\bim_K\) are \(K\)-algebras, morphisms from \(R\) to \(S\) are those \(S\)-\(R\)-bimodules for which left and right \(K\)-module structures coincide, \(2\)-morphisms are bimodule homomorphisms. The composition of bimodules is given by the tensor product and neutral morphisms are the bimodules \(\sub R{R_R}\).

An adjunction in a bicategory \(\mathcal B\) is a pair of morphisms \(f \in \mathcal B(Y, X)\) and \(g \in \mathcal B(X, Y)\) with \(2\)-morphisms \(\eps \colon f \otimes g \rar I_X\) and \(\eta \colon I_Y \rar g \otimes f\) such that \(\eps\) and \(\eta\) satisfy the ordinary unit and counit equations of an adjunction. An adjunction is called an adjoint equivalence if \(\eps\) and \(\eta\) are isomorphisms.

\section{Hermitian bimodules}

Most part of this section is taken from books \cite{Bak} and \cite{Knus}.

A pseudo-involution with symmetry \(\lambda \in R^*\) on a ring \(R\) is a map \(\inv{(-)} \colon R \rar R, a \mapsto \inv a\) such that \(\inv{r + r'} = \inv r + \inv{r'}\), \(\inv 1 = 1\), \(\inv{rr'} = \inv{r'} \inv r\), and \(\inv{\inv r} = \lambda r \inv \lambda\). It is easy to see that in this case \(\inv \lambda = \lambda^{-1}\). If \(\lambda = 1\), then we obtain an involution in the usual sense. If \(R\) is a \(K\)-algebra for commutative ring \(K\) with involution, then we also will assume that \(f(\inv k) = \inv{f(k)}\), where \(f \colon K \rar R\) is the structure homomorphism.

Suppose that \(K\) is a commutative ring with involution. By \(\ibim_K\) we denote the bicategory of \(K\)-algebras with pseudo-involutions and bimodules over them. We omit the index \(K\) in the case \(K = \mathbb Z\).

Let \(R\) and \(S\) be rings with pseudo-involutions and \(\sub S{M_R}\) be a bimodule. One can construct a bimodule \(\sub R{\inv M_S} = \{\inv m \mid m \in M\}\) with operations \(\inv m + \inv{m'} = \inv{m + m'}, \inv r \inv m \inv s = \inv{smr}\). Clearly, there exists coherent isomorphisms
\begin{align*}
\inv{\sub{R_n} {M_n}_{R_{n - 1}} \underset{R_{n - 1}}\otimes \ldots 
\underset{R_1}\otimes \sub{R_1} {M_1}_{R_0}} &\cong \inv{M_1} \underset{R_1}\otimes \ldots \underset{R_{n - 1}}\otimes M_n,\\
\inv{m_n \otimes \ldots \otimes m_1} &\mapsto \inv{m_1} \otimes \ldots \otimes \inv{m_n},
\end{align*}
and they form a contravariant bifunctor \(\ibim_K \rar \ibim_K\) for all \(K\) (which is an identity on objects). In this case we also have isomorphisms \(\inv{\Hom_R(\sub S{M_R}, \sub T{N_R})} \cong \sub R{\Hom(\inv M, \inv N)}\), \(\inv{\sub R{\Hom(\sub R{M_S}, \sub R{N_T})}} \cong \Hom_R(\inv M, \inv N)\). Finally, it is easy to see that there are natural isomorphisms \(\inv{\inv{\sub S{M_R}}} \cong \sub S{M_R}, \inv{\inv m} \mapsto \lambda_S m \inv{\lambda_R}\).

A sesquilinear form \(B\) on a bimodule \(\sub S{M_R}\) is a biadditive map \(B \colon M \times M \rar R\) such that \(B(mr, m'r') = \inv r B(m, m') r'\) and \(B(sm, m') = B(m, \inv sm')\). Equivalently, \(B\) is a bimodule homomorphism \(\inv M \otimes_S M \rar R\). The form \(B\) is called nondegenerate (or regular) if \(M_R\) is finitely generated projective and \(B\) induces a bimodule isomorphism \(\widehat B \colon \inv M \cong \Hom_R(M, R)\).

A sesquilinear form \(B\) is called hermitian if the composition \(\inv M \otimes_S M \cong \inv{\inv M \otimes_S M} \stackrel{\inv B}\rar \inv R \cong R\) equals \(B\). In the language of elements it means that \(B(m, \lambda_S m') = \inv{B(m', m)} \lambda_R\). Such a pair \((M, B)\) will be called a hermitian bimodule (and a hermitian space if \(B\) is nondegenerate).

Let \((\sub S{{M_1}_R}, B_{M_1})\), \ldots, \((\sub S{{M_n}_R}, B_{M_n})\) be hermitian bimodules. The orthogonal sum of them is the hermitian bimodule \(\bigperp_{i = 1}^n (M_i, B_{M_i}) = (\bigoplus_{i = 1}^n M_i, \bigperp_{i = 1}^n B_{M_i})\), where \((\bigperp_{i = 1}^n B_{M_i})((m_i)_{i = 1}^n, (m'_i)_{i = 1}^n) = \sum_{i = 1}^n B_{M_i}(m_i, m'_i)\). It is easy to see that the orthogonal sum is a monoidal operation and the zero bimodule is a neutral element. If \(B_{M_i}\) are nondegenerate, then their orthogonal sum is also nondegenerate.

The tensor product of hermitian bimodules \((\sub{R_0}{M_1}_{R_1}, B_{M_1})\), \ldots, \((\sub{R_{n - 1}}{M_n}_{R_n}, B_{M_n})\) is the hermitian bimodule \(\bigotimes_{i = 1}^n (\sub{R_{i - 1}}{M_i}_{R_i}, B_{M_i}) = (\bigotimes_{i = 1}^n M_i, \bigotimes_{i = 1}^n B_{M_i})\), where \((\bigotimes_{i = 1}^n B_{M_i})(\bigotimes_{i = 1}^n m_i, \bigotimes_{i = 1}^n m'_i) = B_{M_n}(m_n, \ldots B_{M_1}(m_1, m'_1)\ldots m'_n)\) (and the form is biadditive). These tensor products are coherent with neutral elements \((R, B_1)\), where \(B_1(r, r') = \inv rr'\). Hence one can construct the bicategory \(\hbim_K\) of \(K\)-algebras with pseudo-involutions, hermitian bimodules and isometries (a homomorphism \(f \colon M \rar N\) is called an isometry if \(B_N(f(m), f(m')) = B_M(m, m')\)).

Recall that all adjunctions in \(\bim_K\) are of type \((\sub R{P^\vee}_S, \sub S {P_R})\) up to isomorphisms, where \(\eps \colon P^\vee \otimes_S P \rar R\) and \(\eta \colon S \rar \End_R(P) \cong P \otimes_R P^\vee\) are the canonical homomorphisms, \(P_R\) is finitely generated projective and \(P^\vee = \Hom_R(P, R)\). The adjunction is an adjoint equivalence (or Morita equivalence) iff \(P_R\) is fully projective (i.\,e. \(P\) is finitely generated projective and \(R\) is a direct summand of \(P^n\) for \(n\) big enough) and \(S \cong \End_R(P)\).

\begin{lemma}\label{HermitianSpaceProduct}
Let \((\sub T{M_S}, B_M)\) and \((\sub S{N_R}, B_N)\) be hermitian bimodules, \(M_S\) and \(N_R\) are finitely generated projective. If \(B_M\) and \(B_N\) are nondegenerate, then \(B_M \otimes B_N\) is nondegenerate. Conversely, if \(B_M \otimes B_N\) is nondegenerate, \(M_S\) is fully projective, \(N_R\) is fully projective and \(S \cong \End_R(N)\), then \(B_M\) and \(B_N\) are nongenerate.
\end{lemma}
\begin{proof}
Clearly, \((M \otimes_S N)_R\) is finitely generated projective. Note that \(\widehat{B_M \otimes B_N}\) can be expressed as the composition
\begin{align*}
\inv{M \otimes_S N} &\cong \inv N \otimes_S \inv M \stackrel{\widehat B_N \otimes \widehat B_M}\rar \Hom_R(N, R) \otimes_S \Hom_S(M, S) \cong\\
&\cong \Hom_S(M, \Hom_R(N, R)) \cong \Hom_R(M \otimes_S N, R),
\end{align*}
and the first claim follows.

To prove the second claim note that \(\sub S{\inv M}\), \(\sub S{(\Hom_S(M, S))}\), \(\inv N_S\) and \(\Hom_R(N, R)_S\) are faithfully flat. Since the isomorphism \(\widehat B_N \otimes \widehat B_M\) can be presented as \(\inv N \otimes_S \inv M \stackrel{\id \otimes \widehat B_M}\rar \inv N \otimes_S \Hom_S(M, S) \stackrel{\widehat B_N \otimes \id}\rar \Hom_R(N, R) \otimes_S \Hom_S(M, S)\) and as \(\inv N \otimes_S \inv M \stackrel{\widehat B_N \otimes \id}\rar \Hom_R(N, R) \otimes_S \inv M \stackrel{\id \otimes \widehat B_M}\rar \Hom_R(N, R) \otimes_S \Hom_S(M, S)\), both \(\widehat B_M\) and \(\widehat B_N\) are injective and surjective.
\end{proof}

\begin{lemma}\label{HermitianMorita}
Let \(P_R\) be a finitely generated projective module over a ring with pseudo-involution, \(B_P \colon P \times P \rar R\) is a nondegenerate sesquilinear form (if we consider \(P\) as the bimodule \(\sub K{P_R}\)). Then \(S = \End_R(P)\) has unique pseudo-involition such that \(B\) is hermitian, in this case \((P^\vee, P)\) is an adjunction in \(\hbim_K\) (the hermitian form on \(P^\vee = \Hom_R(P, R)\) is uniquely determined). All adjoint equivalences in \(\hbim_K\) are of this type for fully projective \(P\) (up to unique isomorphism).
\end{lemma}
\begin{proof}
Let us construct the pseudo-involition on \(S\). The necessary ans sufficient conditions for \(B_P\) to be a hermitian form are equalities \(B_P(sp, p') = B_P(p, \inv s p')\) and \(B_P(p, \lambda_S p') = \inv{B_P(p', p)} \lambda_R\) for all \(p, p' \in P\). Since \(B_P\) is nondegenerate, for all \(s \in S\) the element \(\inv s\) is uniquely determined from the first equality. It is easy to see that \(s \mapsto \inv s\) is an antihomomorphism of \(K\)-algebras and is consistent with the involution on \(K\). Similarly, the element \(\lambda_S \in S^*\) can be found from the second equality. It remains to prove that this is a pseudo-involution.

From equalities \(B_P(p', \lambda_S sp) = \inv{B_P(sp, p')} \lambda_R = \inv{B_P(p, \inv s p')} \lambda_R = B_P(\inv s p', \lambda_S p)\) it follows that \(\inv{\inv s} \lambda_S = \lambda_S s\). Besides, equalities \(B_P(p, p') = \inv{\inv{B_P(p, p')} \lambda_R} \lambda_R = \inv{B_P(p', \lambda_S p)} \lambda_R = B_P(\lambda_S p, \lambda_S p') = B_P(p, \inv{\lambda_S} \lambda_S p')\) imply \(\inv{\lambda_S} = \lambda_S^{-1}\), hence \(S\) has unique pseudo-involution.

In order to obtain an adjunction, it remains to construct an appropriate hermitian form on \(P^\vee\). Let
\[B_{P^\vee} \colon \inv{P^\vee} \otimes_R P^\vee \cong \sub R{\Hom(\inv P, R)} \otimes_R P^\vee \cong P \otimes_R P^\vee \cong S,\]
i.\,e. \(B_{P^\vee}((p' \mapsto B_P(p \lambda_R, \lambda_S p')), q) = p \otimes q\) in the language of elements. A simple check shows that this is a hermitian form and the pair \((P^\vee, P)\) is an adjunction (where \(\eps \colon P^\vee \otimes_S P \rar R\) and \(\eta \colon S \cong P \otimes_R P^\vee\) are canonical). The uniqueness of such a form on \(P^\vee\) follows from the general uniqueness theorem for an adjoint morphism to a given one.

If \((Q, P)\) is an adjoint equivalence in \(\hbim_K\), then it is also an adjoint equivalence in \(\bim_K\), hence by lemma \ref{HermitianSpaceProduct} both \(P\) and \(Q\) are hermitian spaces. Finally, \(Q \cong P^\vee\) by the uniqueness theorem for an adjoint morphism to a given one.
\end{proof}

Note that lemma \ref{HermitianSpaceProduct} implies that nondegeneracy of hermitian forms is preserved under adjoint equivalences.

\section{Quadratic structures}

Before we define quadratic forms on hermitian bimodules, we should define the objects where quadratic forms maps to. As a motivating example, let us consider the module \(\mathbb F_2^n\) (as a \(\mathbb Z\)-\(\mathbb F_2\)-bimodule) and the quadratic form \(q(b_1, \ldots, b_n) = \sum_{i = 1}^n b_i^2\) on it. In turns out that the values of this form can be considered as elements of \(\mathbb Z / 4 \mathbb Z\) instead of \(\mathbb F_2\), which has application in the coding theory in the definition of doubly even codes.

Let \(R\) be a ring with pseudo-involution. The monoid \(R^\bullet\) acts of \(R\) from the right by the formula \(r \cdot r' = \inv{r'} rr'\). A quadratic structure on \(R\) is a triple \((A_R, \phi, \tr)\), where \(A_R\) is a right \(R^\bullet\)-module and \(\phi \colon R \rar A_R\), \(\tr \colon A_R \rar R\) are \(R^\bullet\)-module homomorphisms that satisfy the following axioms:
\begin{enumerate} % quadratic structure
\item[QS1.] \(\phi(r) = \phi(\inv r \lambda_R)\),
\item[QS2.] \(\tr(\phi(r)) = r + \inv r \lambda_R\),
\item[QS3.] \(\tr(a) = \inv{\tr(a)} \lambda_R\),
\item[QS4.] \(a \cdot (r + r') = a \cdot r + \phi(\inv{r'} \tr(a) r) + a \cdot r'\).
\end{enumerate}

Note that the axiom QS4 is symmetric with respect to \(r\) and \(r'\) under assumptions QS1 and QS3. Besides, axioms QS3 and QS4 follows from QS1 and QS2 for \(a\) in the image of \(\phi\). The axioms imply the useful equality \(\phi(\tr(a)) = a + a \cdot (-1)\).

Below we will be interested not in ordinary rings with pseudo-involutions, but in algebras over a fixed commutative ring \(K\) (also with involution). Hence we need to generalize the notion of algebra for quadratic structures.

A quadratic ring is a commutative ring \(K\) with involution and quadratic structure \(A_K\) such that \(A_K\) itself is a commutative ring (with the additive structure from the \(K^\bullet\)-module structure) satisfying additional axioms
\begin{enumerate} % quadratic ring
\item[QR1.] \(a(a' \cdot k) = aa' \cdot k\),
\item[QR2.] \(a \phi(k) = \phi(\tr(a)k)\),
\item[QR3.] \(\tr(1) = 1\),
\item[QR4.] \(\tr(aa') = \tr(a)\tr(a')\).
\end{enumerate}

Finally, a quadratic algebra over a quadratic ring \((K, A_K)\) is a \(K\)-algebra with pseudo-involution and quadratic structure \((R, A_R)\) such that \(A_R\) is a left \(A_K\)-module (with the same addition as in the \(R^\bullet\)-module structure) satisfying the axioms
\begin{enumerate} % quadratic algebra
\item[QA1.] \(a_K (a_R \cdot r) = a_K a_R \cdot r\),
\item[QA2.] \((a_K \cdot k) a_R = a_K a_R \cdot k\),
\item[QA3.] \(a_K \phi_R(r) = \phi_R(\tr_K(a_K) r)\),
\item[QA4.] \(\phi_K(k) a_R = \phi_R(k \tr_R(a_R))\),
\item[QA5.] \(\tr_R(a_K a_R) = \tr_K(a_K) \tr_R(a_R)\).
\end{enumerate}

Morphisms between quadratic rings are defined in a natural way: morphism \((K, A_K) \rar (L, A_L)\) is a pair of ring homomorphisms \((f \colon K \rar L, g \colon A_K \rar A_L)\) such that \(g(a_K \cdot k) = g(a_K) \cdot f(k)\), \(g(\phi_K(k)) = \phi_L(f(k))\) and \(\tr_L(g(a_K)) = f(\tr_K(a_K))\). Clearly, such a morphism makes \((L, A_L)\) a quadratic algebra over \((K, A_K)\) (in particular, any quadratic ring is a quadratic algebra over itself). Conversely, if \((L, A_L)\) is a quadratic ring and simultaneously a quadratic algebra over \((K, A_K)\), and \(a_K(a_La'_L) = (a_Ka_L)a'_L\) is satisfied, then there exists unique morphism \((K, A_K) \rar (L, A_L)\) that induces this quadratic algebra structure. Moreover, any quadratic algebra over \((L, A_L)\) becomes a quadratic algebra over \((K, A_K)\) through a morphism \((K, A_K) \rar (L, A_L)\) (this is exactly the restriction of scalars).

Let \(K\) be a commutative ring with involution. It turns out that \(K\) has the universal quadratic structure making it a quadratic algebra (more precisely, the initial quadratic structure in the category of all quadratic structures with this property), and all \(K\)-algebras with quadratic structures becomes quadratic algebras over it. Let the Heisenberg group of \(K\) be the set \(\Heis(K) = K \times K\) with the addition operation \((x, y) \dotplus (x', y') = (x + x', y + y' - \inv x x')\) (in general this operation is noncommutative, but it is always a group operation with inverses \(\dotminus (x, y) = (-x, -y - \inv xx)\)). This group has a right action of \(K^\bullet\) given by the formula \((x, y) \cdot k = (xk, \inv kyk)\) and also the multiplication \((x, y) \dottimes (x', y') = (xx', \inv xy'x + \inv{x'}yx' + yy' + \inv yy')\), which is associative with neutral element \(\dot 1 = (1, 0)\). Finally, there are group homomorphisms \(\phi \colon K \rar \Heis(K), k \mapsto (0, k)\) and \(\tr \colon \Heis(K) \rar K, (x, y) \mapsto \inv xx + y + \inv y\).

\begin{lemma}\label{UniversalQuadraticRing}
The group \(A_K = \Heis(K)^{\mathrm{ab}} = \Heis(K) / \{(0, k - \inv k)\}\) makes \(K\) a quadratic ring (all operations defined above are well-defined on \(A_K\)). The pair \((K, A_K)\) is the universal quadratic algebra for \(K\) fixed. All algebras over \(K\) with quadratic structures have unique quadratic algebra structure over \((K, A_K)\).
\end{lemma}
\begin{proof}
It is easy to see that \(\dottimes\) is distributive over \(\dotplus\) from the left, \(\{(0, k - \inv k)\}\) is the commutant of \(\Heis(K)\), and \(\dottimes\) is well-defined on \(A_K = \Heis(K)^{\mathrm{ab}}\). Moreover, \(\dottimes\) is commutative on \(A_K\), hence \(A_K\) is a commutative ring. Clearly, \(\tr(h \dotplus h') = \tr(h) + \tr(h')\) and \(\tr(h \dottimes h') = \tr(h) \tr(h')\), hence \(\tr \colon A_K \rar K\) is well-defined ring homomorphism. It can be checked directly that \(\phi\) and \(\tr\) are \(K^\bullet\)-linear and \((K, A_K)\) satisfies all quadratic ring axioms. If \(A_R\) is a quadratic structure on a \(K\)-algebra \(R\), then axioms QA2 and QA4 imply that \((R, A_R)\) had at most one quadratic algebra structure on \((K, A_K)\) (because \(A_K = 1 \cdot K \dotplus \phi(K)\)), and one can easily prove that this structure is well-defined. The universality property is obvious.
\end{proof}

As a corollary, any ring \(R\) with pseudo-involution and quadratic structure \(A_R\) is a quadratic algebra over \((\mathbb Z, \Heis(\mathbb Z))\) in a unique way. We have the ring isomorphism \(\Heis(\mathbb Z) \cong \mathbb Z[T] / (T^2 - 2T)\), \((x, y) \mapsto x + (y + \binom x 2) T\).

\section{Quadratic forms}

Now we are able to define quadratic forms. Let \((\sub S{M_R}, B)\) be a hermitian bimodule, \((S, A_S)\) and \((R, A_R)\) be quadratic algebras over a quadratic ring \((K, A_K)\). A map \(q \colon A_S \times M \rar A_R\) is called a quadratic form on \((M, B)\), if
\begin{enumerate} % quadratic form
\item[QF1.] \(q(a_S + a'_S, m) = q(a_S, m) + q(a'_S, m)\),
\item[QF2.] \(q(a_S \cdot s, m) = q(a_S, sm)\),
\item[QF3.] \(q(\phi_S(s), m) = \phi_R(B(m, sm))\),
\item[QF4.] \(q(a_S, mr) = q(a_S, m) \cdot r\),
\item[QF5.] \(q(a_S, m + m') = q(a_S, m) + \phi_R(B(m, \tr_S(a_S) m')) + q(a_S, m')\),
\item[QF6.] \(\tr_R(q(a_S, m)) = B(m, \tr_S(a_S)m)\),
\item[QF7.] \(q(a_Ka_S, m) = a_Kq(a_S, m)\).
\end{enumerate}

The triple \((M, B, q)\) will be called quadratic bimodule (and quadratic space, if \(B\) is nondegenerate). Note that QF7 is always satisfied in case \(A_K = \Heis(K)^{\mathrm{ab}}\). Moreover, the axioms imply the relation \(q(a_S, m) + q(a_S, -m) = \phi_R(B(m, \tr_S(a_S), m))\)

Now let us explain the relation between our definition and odd quadratic forms from paper \cite{OddPetrov}. First of all, we will reformulate our definition in the case of a right hermitian module \((M_R, B)\), where \(R\) is a quadratic algebra over \(K\) and \(B\) is a hermitian form on \(\sub K{M_R}\).

\begin{lemma}\label{RightModuleForm}
There is a canonical one-to-one correspondence between quadratic forms on a right hermitian module \((M_R, B)\) and maps \(\widetilde q \colon M \rar A_R\) such that \(\widetilde q(mr) = \widetilde q(m) \cdot r\), \(\widetilde q(m + m') = \widetilde q(m) + \phi_R(B(m, m')) + \widetilde q(m')\), and \(\tr_R(\widetilde q(m)) = B(m, m)\).
\end{lemma}
\begin{proof}
Recall that we consider any right module \(M_R\) as the bimodule \(\sub K{M_R}\), where \(K\) is the base commutative ring. It is easy to see that if \(q\) is a quadratic form on the module \(\sub K{M_R}\), then the map \(\widetilde q \colon m \mapsto q(1, m)\) satisfies the required identities. Conversely, any such a map \(\widetilde q\) is obtained in this way from unique form \(q\) given by the formula \(q(a_K, m) = a_K \widetilde q(m)\).
\end{proof}

In particular, this lemma implies that when we work with right modules we can forget about \(K\) and work only with the ring \(R\) with quadratic structure \(A_R\). It turns out that in some sense one can describe all possible structures and quadratic forms for a fixed module \(M_R\).

Recall the notions from paper \cite{OddPetrov}. If \(M_R\) is a module with a hermitian form \(B\), then its Heisenberg group \(\Heis(B)\) is the set \(M \times R\) with the operation \((m, r) \dotplus (m', r') = (m + m', r + r' - B(m, m'))\) (it is a group operation with inverses \(\dotminus (m, r) = (-m, -r - B(m, m))\)). The group \(\Heis(B)\) has a right action of \(R^\bullet\) given by the formula \((m, r) \cdot r' = (mr', \inv{r'}rr')\). Moreover, there are \(R^\bullet\)-linear maps \(\phi \colon R \rar \Heis(B), r \mapsto (0, r)\) and \(\tr \colon \Heis(B) \rar R, (m, r) \mapsto B(m, m) + r + \inv r \lambda_R\). Finally, there is a map \(\widetilde q \colon M \rar \Heis(B), m \mapsto (m, 0)\). A subgroup \(\mathcal L \leq \Heis(B)\) is called an odd form parameter, if it is closed under the action of \(R^\bullet\) and \(\mathcal L_{\min} \leq \mathcal L \leq \mathcal L_{\max}\), where \(\mathcal L_{\min} = \{(0, r - \inv r \lambda_R)\}\), \(\mathcal L_{\max} = \Ker(\tr)\) (it is true that \(\mathcal L_{\min} \leq \mathcal L_{\max}\), these subgroups of \(\Heis(B)\) are closed under the action of \(R^\bullet\), and \(\mathcal L_{\min}\) contains the commutant \(\Heis(B)\)).

\begin{lemma}\label{HeisenbergGroup}
The abelian group \(\Heis(B) / \mathcal L\) is a quadratic structure, a pair \((R, \Heis(B) / \mathcal L)\) is a quadratic algebra over \((K, A_K)\), where \(A_K = \Heis(K)^{\mathrm{ab}}\), and \(\widetilde q\) induces unique quadratic form \(q\) on \(\sub K{M_R}\) such that \(q(1, m) = \widetilde q(m) \in \Heis(B) / \mathcal L\). Conversely, any quadratic form \(q\) on \(\sub K{M_R}\) with a quadratic structure on \(R\) can be obtained in this way (uniquely up to unique isomorphism), if \(A_R\) is generated by the images of \(q\) and \(\phi_R\) as an abelian group.
\end{lemma}
\begin{proof}
All claims can be checked directly. The quadratic form \(q\) can be constructed by the explicit formula \[q((x, y), m) = \widetilde q(xm) + \phi(B(m, ym)).\]
The odd form parameter \(\mathcal L\) can be recovered as the kernel of \(R^\bullet\)-linear map \(\Heis(B) \rar A_R, (m, r) \mapsto \tilde q(1, m) + \phi(r)\), and the same map gives the isomorphism between \(\Heis(B) / \mathcal L\) and \(A_R\).
\end{proof}

If \((M_1, B_{M_1}, q_{M_1})\), \ldots, \((M_n, B_{M_n}, q_{M_n})\) are quadratic bimodules over \((S, A_S)\) and \((R, A_R)\), the their orthogonal sum is the quadratic bimodule \(\bigperp_{i = 1}^n(M_i, B_{M_i}, q_{M_i}) = (\bigoplus_{i = 1}^n M_i, \bigperp_{i = 1}^n B_{M_i}, \bigperp_{i = 1}^n q_{M_i})\), where the quadratic form is given by the formula \((\bigperp_{i = 1}^n q_{M_i})(a_S, (m_i)_{i = 1}^n) = \sum_{i = 1}^n q_{M_i}(a_S, m_i)\). Clearly, this is a monoidal operation (the zero bimodule is a neutral element). The orthogonal sum of quadratic spaces is a quadratic space.

One can also define the tensor product of quadratic bimodules. If \((R_0, A_{R_0})\), \ldots, \((R_n, A_{R_n})\) are quadratic algebras and \((\sub{R_0}{M_1}_{R_1}, B_{M_1}, q_{M_1})\), \ldots, \((\sub{R_{n - 1}}{M_n}_{R_n}, B_{M_n}, q_{M_n})\) are quadratic bimodules, then their tensor product is \(\bigotimes_{i = 1}^n (M_i, B_{M_i}, q_{M_i}) = (\bigotimes_{i = 1}^n M_i, \bigotimes_{i = 1}^n B_{M_i}, \bigotimes_{i = 1}^n q_{M_i})\), where
\begin{align*}
(\bigotimes_{i = 1}^n q_{M_i})(a_{R_0}, &\sum_{1 \leq l \leq L} \bigotimes_{i = 1}^n m_{l, i}) = \sum_{1 \leq l \leq L} q_{M_n}(\ldots q_{M_1}(a_{R_0}, m_{l, 1}), \ldots, m_{l, n}) + \phantom{a}\\
&+ \sum_{1 \leq l < l' \leq L} \phi_{R_n}(B_{M_n}(m_{l, n}, \ldots B_{M_1}(m_{l, 1}, \tr_{R_0}(a_{R_0})m_{l, 1}) \ldots m_{l', n})).
\end{align*}

This tensor product is associative up to isomorphism with neutral elements \((R, B_1, q_1)\), where \(q_1(a_R, r) = a_R \cdot r\). Therefore we obtain a bicategory \(\qbim_K\), where objects are quadratic algebras over a quadratic ring \((K, A_K)\), morphisms are quadratic bimodules, and \(2\)-morphisms are isometries (which preserve both hermitian and quadratic forms). Note that morphisms \((K, A_K) \rar (L, A_L)\) induces the scalar restriction bifunctors \(\qbim_L \rar \qbim_K\).

\begin{prop}\label{QuadraticMorita}
Let \((P_R, B_P)\) be a finitely generated projective module with a nondegenerate sesquilinear form, \((R, A_R)\) be a quadratic algebra. Then \(S = \End_R(P)\) has unique (up to unique isomorphism) quadratic structure with a quadratic form \(q_P\) on \(P\) such that \((P, B_P, q_P)\) has a left adjoint in \(\qbim_K\). All adjoint equivalences are obtained in this way (up to unique isomorphism).
\end{prop}
\begin{proof}
The pseudo-involution on \(S\) is constructed uniquely by lemma \ref{HermitianMorita}, and this lemma also implies the second claim (after proving uniqueness in the first one). Let \((Q, B_Q)\) be the left adjoint to \((P, Q_P)\) in \(\hbim_K\) (from now on \(P\) is an \(S\)-\(R\)-bimodule). Note that adjointness for quadratic forms means \(a_S \cdot (p \otimes q) = q_Q(q_P(a_S, p), q)\) and \(a_R \cdot \langle q, p \rangle = q_P(q_Q(a_R, q), p)\).

Let us start from existence. Let
\[A_S = \frac{\langle q_Q(a_R, q), \phi_S(s) \mid a_R \in A_R, q \in Q, s \in S \rangle_{\mathbb Z}}{\left\langle \substack{\phi_S(s + s') - \phi_S(s) - \phi_S(s'),\enskip \phi_S(s - \inv s \lambda_S),\enskip q_Q(\phi_R(r), q) - \phi_S(B_Q(q, rq)),\\ q_Q(a_R + a'_R, q) - q_Q(a_R, q) - q_Q(a'_R, q),\enskip q_Q(a_R \cdot r, q) - q_Q(a_R, rq),\\ q_Q(a_R, q + q') - q_Q(a_R, q) - q_Q(a_R, q') - \phi_S(B_Q(q, \tr_R(a_R)q'))} \right\rangle},\]
where \(q_Q(a_R, q)\) and \(\phi_S(s)\) are formal symbols (from now on we can consider them as values of the maps \(q_Q \colon A_R \times Q \rar A_S\) and \(\phi_S \colon S \rar A_S\)). The abelian group \(A_S\) has a right \(S^\bullet\)-module structure given by the formulas \(\phi_S(s) \cdot s' = \phi_S(\inv{s'}ss')\) and \(q_Q(a_R, q) \cdot s = q_Q(a_R, qs)\) (this structure is well-defined). Moreover, let us define an additive map \(\tr_S \colon A_S \rar S\) by the formulas \(\tr_S(\phi_S(s)) = s + \inv s \lambda_S\) and \(\tr_S(q_Q(a_R, q)) = B_Q(q, \tr_R(a_R)q)\) (this map is also well-defined). Finally, \(A_S\) has a well-defined structure of a left \(A_K\)-module given by \(a_K \phi_S(s) = \phi_S(\tr_K(a_K)s)\) and \(a_K q_Q(a_R, q) = q_Q(a_Ka_R, q)\).

One can directly check that \((S, A_S)\) is a quadratic algebra and \(q_Q\) is a quadratic form. The quadratic form on the bimodule \(P\) is given by the formulas \(q_P(\phi_S(s), p) = \phi_R(B_P(p, sp))\) and \(q_P(q_Q(a_R, q), p) = a_R \cdot \langle q, p \rangle\), it is also well-defined. Another check shows that this is a quadratic form and \((Q, B_Q, q_Q)\) is the left adjoint to \((P, B_P, q_P)\).

Finally, let us prove uniqueness. Let \(A'_S\) be another quadratic structure on \(S\), \(q'_Q\) and \(q'_P\) be the corresponding quadratic forms. Clearly, there exists a natural map \(\pi \colon A_S \rar A'_S\) such that all needed diagrams are commutative, and we have to prove that it is one-to-one. Note that a \((Q, A_S)\)-\((Q, A'_S)\)-bimodule of type \((S, B_1, q)\) with some quadratic form \(q\) is the same as a map \(f \colon A_S \rar A'_S\) making the obvious diagrams commutative (it is true for any ring \(S\) with two quadratic structures). The map \(\pi\) is obtained from the bimodule \((P, B_P, q_P) \otimes_R (Q, B_Q, q'_Q)\) in this way. But this bimodule has a quasi-inverse \((P, B_P, q'_P) \otimes_R (Q, B_Q, q_Q)\) (their composition in any order has an isometry into the identity morphism, which is an isomorphism in \(\hbim_K\), and, consequently, in \(\qbim_K\)). Therefore, \(\pi\) is an isomorphism (and it is clearly unique if we require commutativity of all necessary diagrams).
\end{proof}

\section{Tensor products}

Here we will give definitions of tensor products of quadratic algebras and their localizations with respect to a multiplicative subset of the base ring \(K\). These notions are given in the sake of completeness and will not be used further in the paper.

Let \(R\) and \(S\) be algebras with pseudo-involutions over \(K\). Their tensor product \(R \otimes_K S\) has a pseudo-involution \(\inv{r \otimes s} = \inv r \otimes \inv s\) with symmetry \(\lambda_R \otimes \lambda_S\). This operation is monoidal on \(\ibim_K\) with neutral element \(K\). For any commutative algebra \(L\) with involution we have a well-defined functor \(L \otimes_K - \colon \ibim_K \rar \ibim_L\) (which can be called the scalar extension functor), it preserves the tensor product of algebras up to isomorphism. In particular, if \(S = \inv S \subseteq K\) is a multiplicative subset and \(L = S^{-1}K\) we have \(L \otimes_K R = S^{-1}R\), that is the scalar extension coincides with the localization. Note that the localization with respect to such \(S\) is the same as the localization with respect to \(\{s \inv s \mid s \in S\}\), hence we can consider only localizations with respect to subsets of \(\herm(K)\).

One can also define tensor products over \(K\) and scalar extension functors for the bicategory \(\hbim_K\). The hermitian form on the tensor product of \((\sub SM_R, B_M)\) and \((\sub TN_U, B_N)\) is given by the obvious formula \(B_{M \otimes_K N}(m \otimes n, m' \otimes n') = B_M(m, m') \otimes B_N(n, n')\). It is easy to see that tensor products and scalar extensions preserve nondegeneracy of forms. Next we will generalize this to \(\qbim_K\).

Let \((R, A_R)\) and \((S, A_S)\) be quadratic algebras on \((K, A_K)\). Their tensor product is \((R, A_R) \otimes_{(K, A_K)} (S, A_S) = (T, A_T)\), where \(T = R \otimes_K S\) and
\[A_T = \frac{\langle a_R \otimes a_S, \phi(t) \mid a_R \in A_R, a_S \in A_S, t \in T \rangle_{\mathbb Z}}{\left \langle \substack{(a_R + a'_R) \otimes a_S - a_R \otimes a_S - a'_R \otimes a_S,\enskip a_R \otimes (a_S + a'_S) - a_R \otimes a_S - a_R \otimes a'_S,\\ a_K a_R \otimes a_S - a_R \otimes a_K a_S,\enskip \phi(t + t') - \phi(t) - \phi(t'),\enskip \phi(t - \inv t \lambda_T),\\ \phi(r \otimes \tr_S(a_S)) - \phi_R(r) \otimes a_S,\enskip \phi(\tr_R(a_R) \otimes s) - a_R \otimes \phi_S(s)} \right \rangle}.\]

This \(A_T\) is a right \(T^\bullet\)-module with the multiplication
\[(a_R \otimes a_S) \cdot (\sum_{\!\!\!1 \leq i \leq n\!\!\!} r_i \otimes s_i) = \sum_{\!\!\!1 \leq i \leq n\!\!\!} (a_R \cdot r_i) \otimes (a_S \cdot s_i) + \phi(\sum_{\!\!\!\!\!\!\!\!1 \leq i < j \leq n\!\!\!\!\!\!\!\!} \inv{r_i} \tr_R(a_R) r_j \otimes \inv{s_i} \tr_S(a_S) \otimes s_j),\]
\(\phi(t) \cdot t' = \phi(\inv{t'}tt')\) and a left \(A_K\)-module with the multiplication \(a_K(a_R \otimes a_S) = a_Ka_R \otimes a_S\), \(a_K \phi(t) = \phi(\tr_K(a_K)t)\). Let \(\phi_T(t) = \phi(t)\), \(\tr_T(a_R \otimes a_S) = \tr_R(a_R) \otimes \tr_S(a_S)\), \(\tr_T(\phi(t)) = t + \inv t \lambda_T\). A direct check shows that these operations are well-defined and \((T, A_T)\) is actually a quadratic algebra. If \((\sub S{M_R}, B_M, q_M)\) and \((\sub T{N_U}, B_N, q_N)\) are quadratic bimodules, then we can construct the quadratic bimodule \((\sub{S \otimes T}{(M \otimes N)}_{R \otimes U}, B_{M \otimes N}, q_{M \otimes N})\), where
\begin{align*}
q_{M \otimes N}(a_S \otimes a_T)(\sum_{\!1 \leq i \leq n\!} m_i \otimes n_i) &= \sum_{\!1 \leq i \leq n\!}q_M(a_S, m_i) \otimes q_N(a_T, n_i) +\phantom{a}\\
&+ \phi(\sum_{\!\!\!1 \leq i < j \leq n\!\!\!} B_M(m_i, \tr_S(a_S)m_j) \otimes B_N(n_i, \tr_T(a_T)n_j))
\end{align*}
and \(q_{M \otimes N}(\phi(t), \sum_{1 \leq i \leq n} m_i \otimes n_i) = \phi(B_{M \otimes N}(\sum_{1 \leq i \leq n} m_i \otimes n_i, t \, \sum_{1 \leq j \leq n} m_j \otimes n_j))\). Therefore, we obtain the bifunctor \(\qbim_K \times \qbim_K \rar \qbim_K\), which is commutative and associative up to isomorphism, \((K, A_K)\) is a neutral element.

Now let us consider scalar extensions. If \((L, A_L)\) is a quadratic ring and \((K, A_K) \rar (L, A_L)\) is a morhism, then there is a functor \((L, A_L) \otimes_{(K, A_K)} - \colon \qbim_K \rar \qbim_L\). Indeed, for any quadratic algebra \((R, A_R)\) over \((K, A_K)\) the tensor product \((L, A_L) \otimes_{(K, A_K)} (R, A_R)\) is a quadratic algebra on \((L, A_L)\), the left \(A_L\)-module structure on \(A_{L \otimes_K R}\) is given by the formulas \(a_L (a'_L \otimes a_R) = (a_La'_L) \otimes a_R\) and \(a_L \phi(t) = \phi(\tr_L(a_L) t)\). It is easy to check that this construction is well-defined and \((L \otimes_K R, A_{L \otimes_K R})\) is a quadratic algebra on \((L, A_L)\). Moreover, any bimodule \((L, B_1, q_1) \otimes (M, B_M, q_M)\) is automatically a morphism in \(\qbim_L\). Note that the scalar extension is functorial on \((K, A_K) \rar (L, A_L)\) and preserves tensor products of quadratic algebras.

It is easy to see that the tensor product \((L, A_L) \otimes_{(K, A_K)} (O, A_O)\) of quadratic rings is a pushout. Hence the category of quadratic rings is finitely cocomplete with an initial object \((\mathbb Z, \Heis(\mathbb Z))\).

Finally, we consider localizations. Let \((K, A_K)\) be a quadratic ring, \(S = \inv S \subseteq K\) be a multiplicative subset, \((R, A_R)\) be a quadratic algebra over \((K, A_K)\). The localization of \((R, A_R)\) with respect to \(S\) is the pair \((S^{-1} R, A_{S^{-1} R})\), where \(A_{S^{-1} R} = (1 \cdot S)^{-1} A_R\) (note that \(1 \cdot S\) is a multiplicative subset in \(A_K\)). Obviously, this pair is a quadratic algebra on \((K, A_K)\) with the operations \(\frac{a_R}{1 \cdot s} \cdot \frac{r}{s'} = \frac{a_R \cdot r}{1 \cdot ss'}\), \(\phi_{S^{-1} R}(\frac rs) = \frac{\phi_R(r\inv s)}{1 \cdot s}\), and \(\tr_{S^{-1} R}(\frac{a_R}{1 \cdot s}) = \frac{\tr_R(a_R)}{s\inv s}\). The quadratic algebra \((S^{-1} K, A_{S^{-1} K})\) is actually a quadratic ring. Moreover, there is an isomorphism \((S^{-1} R, A_{S^{-1} R}) \cong (S^{-1} K, A_{S^{-1} K}) \otimes_{(K, A_K)} (R, A_R)\). The localization on quadratic bimodules is given by \(S^{-1}(\sub R{M_T}, B_M, q_M) = (S^{-1} M, B_{S^{-1} M}, q_{S^{-1} M})\), where \(q_{S^{-1} M}(\frac{r}{1 \cdot s}, \frac m{s'}) = \frac{q_M(r, m)}{1 \cdot ss'}\).

At the end let us make a remark. If \(K = Kf_1 + \ldots + K f_n\), where \(f_1, \ldots, f_n \in \herm(K)\), then \(A_K = A_K (1 \cdot f_1) + \ldots + A_K (1 \cdot f_n)\). Indeed, we have \(\sum_{i = 1}^n k_i f_i = 1\) and after simplifying \(1 = 1 \cdot (\sum_{i = 1}^n k_i f_i)^N\) for \(N\) big enough all summands of type \(\phi_K(\ldots)\) will have at least two identical factors \(f_i\) inside the brackets, therefore they can be pushed outside by the formula \(\phi_K(gf_i^2) = \phi_K(g) \cdot f_i\). Then one can, for example, construct sheaves of quadratic algebras on \(\mathrm{Spec}(\herm(K))\).

\section{Elementary transvections}

Let us modify several standard definitions for hermitian bimodules in the our context. If \((\sub T{P_R}, B_P, q_P)\) is a quadratic bimodule and \(P_R\) is finitely generated projective, then the quadratic bimodule \((\met(P), B_{\met(P)}, q_{\met(P)})\) is called the metabolic space over \(P\), where \(\met(P) = P \oplus \inv{P^\vee}\) is a \(T\)-\(R\)-bimodule, \(B_{\met(P)}(\veccol p{\inv f}, \veccol{p'}{\inv{f'}}) = \lambda_R f(\inv \lambda_T p') + \inv{f'(p)} + B_P(p, p')\) and \(q_{\met(P)}(a_T, \veccol p{\inv f}) = \phi_R(\lambda_Rf(\inv{\lambda_T} \tr_T(a_T) p)) + q_P(a_T, p)\). This bimodule is actually a quadratic space. In the case \(B_P = q_P = 0\) the metabolic space is called the hyperbolic space and is denoted by \(\hyp(P)\).

Let \((\sub T{P_R}, B_P, q_P)\) be a quadratic space. Its automorphism group is called the unitary group and is denoted by \(\unit(P) = \unit(P, B_P, q_P)\). We will be interested in the case \(P = P_0 \perp \hyp(P_1) \perp \ldots \perp \hyp(P_l)\), where \(P_0\) is a quadratic space (the odd part of \(P\)). Let \(P_{-i} = \inv{P_i^\vee}\) for all \(1 \leq i \leq l\), then \(P = \bigoplus_{i = -l}^l P_i\) as a \(T\)-\(R\)-bimodule. From now on we will assume that for all \(i, j \neq 0\) the bimodule \(P_i\) is isomorphic to a direct summand in \(P_j^N\) for \(N\) big enough.

Let \(E = \End_R(P_R)\), then by proposition \ref{QuadraticMorita} there is an adjunction \((P^\vee, P)\) in \(\qbim_K\). In this case \(E\) is a quadratic space (as a bimodule \(\sub T{E_E}\)) and the adjunction gives the isomorphism \(\Aut_\bim(E) \cong \Aut_\bim(P) = \gl(P)\) (since \(E \otimes_E P \otimes_R P^\vee \cong E\) and \(P \otimes_R P^\vee \otimes_E P \cong P\)), which induces the isomorphism \(\unit(E) \cong \unit(P)\). Let \(C = \Cent_E(T) = \End_\bim(E) = \End_\bim(P)\), it is a \(K\)-algebra with involution (that is obtained as the restriction of the pseudo-involution on \(E\)). The algebra \(C\) commutes with \(\lambda_E\), because \(\lambda_E\) equals to the image of \(\lambda_T\) by lemma \ref{HermitianMorita}). Then \(\Aut_\bim(E) = \Aut_\bim(P) = C^*\).

The condition \(P = P_0 \perp \hyp(P_1) \perp \ldots \perp \hyp(P_l)\) is equivalent to existence of a complete family of pairwise orthogonal idempotents \(\{e_i\}_{-l \leq i \leq l}\) in \(C\) such that \(\inv{e_i} = e_{-i}\), \(q(A_T, e_i) = 0\) for all \(i \neq 0\), and our additional assumption is equivalent to \((1 - e_0) C e_i C (1 - e_0) = (1 - e_0) C (1 - e_0)\) for all \(i \neq 0\).

Let
\[t_{i, j}(x) = 1 + x\]
for all \(x \in e_i C e_j\) and \(i \neq j\), similarly to the elementary transvections in the general linear group. It is easy to check the identities
\begin{enumerate} % linear transvections
\item[LT1.] \(t_{i, j}(x) \, t_{i, j}(y) = t_{i, j}(x + y)\);
\item[LT2.] \([t_{i, j}(x), t_{j, k}(y)] = t_{i, k}(xy)\) for all \(i \neq k\);
\item[LT3.] \([t_{j, i}(x), t_{k, j}(y)] = t_{k, i}(-yx)\) for all \(i \neq k\);
\item[LT4.] \([t_{i, j}(x), t_{k, l}(y)] = 1\) for all \(j \neq k\) and \(l \neq i\).
\end{enumerate}

The elementary transvections are
\[\tau_{i, j}(x, y) = t_{i, j}(x) t_{-j, -i}(-\inv y) = 1 + x - \inv y\]
for all \(x, y \in e_i C e_j\), \(0 \neq i \neq \pm j \neq 0\) and
\[\tau_i(x, y, z) = t_{0, i}(x) t_{-i, i}(y) t_{-i, 0}(z) = 1 + x + y + z\]
for all \(x \in e_0 C e_i\), \(y \in e_{-i} C e_i\), \(z \in e_{-i} C e_0\), \(i \neq 0\). The elementary unitary group is the group
\[\eunit(P) = \langle \tau_{i, j}(x, y), \tau_i(x, y, z) \mid \tau_{i, j}(x, y) \in \unit(P), \tau_i(x, y, z) \in \unit(P) \rangle \leq \unit(P).\]

For comfort work with \(\tau_{i, j}(x, y)\) and \(\tau_i(x, y, z)\) let us introduce the \(K\)-algebra \(\mathcal A = C \times C\) with the diagonal embedding \(C \hookrightarrow A\) and the involution \(\inv{(x, y)} = (\inv y, \inv x)\), and the group \(\mathcal H = e_0 C \times C \times C e_0\) with the group operation \((x, y, z) \dotplus (x', y', z') = (x + x', y + zx' + y', z + z')\), the neutral element \(\dot 0 = (0, 0, 0)\) and the inverses \(\dotminus (x, y, z) = (-x, zx - y, -z)\). There is the right action of \(\mathcal A^\bullet\) on \(\mathcal H\) given by the formula \((x, y, z) \cdot (p, q) = (xp, \inv q y p, \inv q z)\), and there are the group homomorphisms \(\pi \colon \mathcal H \rar \mathcal A, (x, y, z) \mapsto (x, -\inv z)\) and \(\phi \colon \mathcal A \rar \mathcal H, (x, y) \mapsto (0, x - \inv y, 0)\). Finally, there is the map \(\tr \colon \mathcal H \rar \mathcal A, (x, y, z) \mapsto (y, \inv{zx - y})\). All these operations satisfy
\begin{enumerate} % noncommutative quadratic something
\item[NQ1.] \(h \dotplus h' \dotminus h \dotminus h' = \phi(-\inv{\pi(h)} \pi(h'))\), \(h \dotplus \phi(a) = \phi(a) \dotplus h\);
\item[NQ2.] \(\phi(\inv a) = \dotminus \phi(a) = \phi(-a)\);
\item[NQ3.] \(\pi(\phi(a)) = 0\);
\item[NQ4.] \(\tr(h \dotplus h') = \tr(h) - \inv{\pi(h)} \pi(h') + \tr(h')\), \(\tr(\dot 0) = 0\), \(\tr(\dotminus h) = -\inv{\pi(h)} \pi(h) - \tr(h)\);
\item[NQ5.] \(\inv{\tr(h)} = \tr(\dotminus h)\);
\item[NQ6.] \(h \cdot (a + a') = h \cdot a \dotplus \phi(\inv{a'} \tr(h) a) \dotplus h \cdot a'\), \(h \cdot 0 = \dot 0\), \(h \cdot (-1) = \phi(\tr(h)) \dotminus h\);
\item[NQ7.] \(\phi(a) \cdot a' = \phi(\inv{a'} aa')\);
\item[NQ8.] \(\tr(h \cdot a) = \inv a \tr(h) a\);
\item[NQ9.] \(\pi(h \cdot a) = \pi(h) a\);
\item[NQ10.] \(\tr(\phi(a)) = a - \inv a\), \(\phi(\tr(h)) = h \dotplus h \cdot (-1)\).
\end{enumerate}

\begin{lemma}\label{LevelTransvections}
There are the following identities:
\begin{enumerate} % transvections
\item[T1.] \(\tau_{i, j} \colon e_i \mathcal A e_j \rar \gl(P)\) and \(\tau_i \colon \mathcal H \cdot e_i \rar \gl(P)\) are homomorphisms;
\item[T2.] \(\tau_{i, j}(a) = \tau_{-j, -i}(-\inv a)\);
\item[T3.] \([\tau_{i, j}(a), \tau_{k, l}(a')] = 1\), if \(i \neq l \neq -j \neq -k \neq i\);
\item[T4.] \([\tau_{i, j}(a), \tau_{j, k}(a')] = \tau_{i, k}(aa')\) and \([\tau_{j, i}(a), \tau_{k, j}(a')] = \tau_{k, i}(-a'a)\), if \(i \neq \pm k\);
\item[T5.] \([\tau_{-i, j}(a), \tau_{j, i}(a')] = \tau_i(\phi(aa'))\);
\item[T6.] \([\tau_i(h), \tau_j(h')] = \tau_{-i, j}(-\inv{\pi(h)} \pi(h'))\), if \(i \neq \pm j\);
\item[T7.] \([\tau_i(h), \tau_{j, k}(a)] = 1\), if \(j \neq i \neq -k\);
\item[T8.] \([\tau_i(h), \tau_{i, j}(a)] = \tau_{-i, j}(\tr(h) a) \tau_j(\dotminus h \cdot (-a))\).
\end{enumerate}
\end{lemma}
\begin{proof}
This directly follows from the relations for \(t_{i, j}(x)\).
\end{proof}

\section{Level groups}

Let \(\Lambda_i = \{h \in \mathcal H \cdot e_i \mid \tau_i(h) \in \unit(P)\}\) for \(i \neq 0\). This group depends on \(q\), for any \((x, y, z) \in \Lambda_i\) there are the equalities \(z = -\inv x\) and \(y + \inv y = zx\) (they are equivalent to \(\inv{\tau_i(x, y, z)} = \tau_i(-x, zx - y, -z)\)). Let also \(\Lambda = \sum_{i \neq 0}^\cdot \Lambda_i \dotplus \phi(C) \cdot (1 - e_0)\).

\begin{lemma}\label{Lambda}
Let \(l \geq 3\). Then \(\pi(\Lambda) \leq e_0 C (1 - e_0)\), \(\tr(\Lambda) + \inv{\pi(\Lambda)} \pi(\Lambda) \leq (1 - e_0) C (1 - e_0)\), and \(\Lambda \cdot (1 - e_0) C (1 - e_0) \dotplus \phi(C) \cdot (1 - e_0) \leq \Lambda\). Moreover, \(\Lambda_i = \Lambda \cdot e_i\) for all \(i \neq 0\).
\end{lemma}
\begin{proof}
This is followed from lemma \ref{LevelTransvections}, because \(\tau_{i, j}(a) \in \unit(P)\) iff \(a \in e_i C e_j\).
\end{proof}

As an example let us find explicitly \(\Lambda\) in the case \(T = K\) and the the quadratic form is given by an odd form parameter \(\mathcal L \leq \Heis(B)\). The condition \(q(1, e_i) = 0\) means \((e_i, 0) \in \mathcal L\) for all \(0 \neq i\). The element \(\tau_i(x, y, z)\) is in \(\unit(P)\), if \(z = -\inv x\) and \((xp, B(p, yp)) \in \mathcal L\) for all \(p \in P_i\). Therefore, \(\Lambda_i = \{(x, y, -\inv x) \in \mathcal H \cdot e_i \mid (xp, B(p, yp)) \in \mathcal L \text{ for } p \in \bigoplus_{i \neq 0} P_i\}\). Conversely, if \(\Lambda\) satisfies all conditions from lemma \ref{Lambda}, \(T = K\) and \(P = E = C\), then \(\Lambda\) is obtained from the odd form parameter \(\mathcal L = \{(x, y) \mid (x, y, -\inv x) \in \Lambda\} \cdot C \dotplus \langle (e_i, 0) \mid i \neq 0 \rangle \dotplus \mathcal L_{\mathrm{min}}\).

An augmented level is a pair \(L = (I, \Gamma)\), where \(I = \inv I \leq \mathcal A\), \(\Gamma \leq \mathcal H\) and \(I \, (I + K + (1 - e_0) C (1 - e_0) + \pi(\Lambda) + \inv{\pi(\Lambda)}) \leq I\), \(\pi(\Gamma) \leq I\), \(\tr(\Gamma) + \inv{\pi(\Gamma)} \pi(\Gamma) \leq I\), \(\Gamma \cdot (I + K + (1 - e_0) C (1 - e_0) + \pi(\Lambda) + \inv{\pi(\Lambda)}) \dotplus \Lambda \cdot I \dotplus \phi(I) \leq \Gamma\), \(e_0 I (1 - e_0) = \pi(\Gamma \cdot (1 - e_0))\). Two augmented levels \(L = (I, \Gamma)\) and \(L' = (I', \Gamma')\) are called equivalent if \(I(1 - e_0) = I'(1 - e_0)\) and \(\Gamma \cdot (1 - e_0) = \Gamma' \cdot (1 - e_0)\), the equivalence classes are called levels. We will often denote the class of an augmented level \(L\) also as \(L\). For a level \(L\) its enveloping level is the level \(\widehat L\), where \(\widehat I (1 - e_0) = I(1 - e_0) + (1 - e_0) C (1 - e_0) + \pi(\Lambda)\) and \(\widehat \Gamma \cdot (1 - e_0) = \Gamma \cdot (1 - e_0) \dotplus \Lambda\). Any level \(L\) as a class contains the smallest and the biggest augmented levels \(\lfloor L \rfloor\) and \(\lceil L \rceil\), where \(e_0 \lfloor I \rfloor e_0 = e_0 I (1 - e_0) \widehat I e_0 + e_0 \widehat I (1 - e_0) I e_0\), \(\lfloor \Gamma \rfloor \cdot e_0 = \Gamma \cdot (1 - e_0) \widehat I e_0 \dotplus \Lambda \cdot Ie_0 \dotplus \phi(\lfloor I \rfloor) \cdot e_0\), \(e_0 \lceil I \rceil e_0 = \{a \in e_0 \mathcal A e_0 \mid a \widehat I (1 - e_0) + (1 - e_0) \widehat I a \leq I\}\), and \(\lceil \Gamma \rceil \cdot e_0 = \{h \in \mathcal H \cdot e_0 \mid \pi(h) \in \lceil I \rceil, \tr(h) \in \lceil I \rceil, h \cdot \widehat I (1 - e_0) + \phi(I) \cdot (1 - e_0) \leq \Gamma\}\).

\begin{lemma}\label{GroupLevel}
Let \(l \geq 3\) and \(G \leq \gl(P)\) is normalized by \(\eunit(P)\). Then there exists unique level \(L(G) = (I(G), \Gamma(G))\) such that \(\tau_{i, j}(a) \in G\) iff \(a \in I(G)\) and \(\tau_i(h) \in G\) iff \(h \in \Gamma(G)\).
\end{lemma}
\begin{proof}
The proof is completely similar to the one of lemma \ref{Lambda}, if we note that \([G, G] \leq G\) and \([G, \eunit(P)] \leq G\).
\end{proof}

Therefore, any group \(G\) normalized by \(\eunit(P)\) corresponds to the level \(L(G)\). If \(L = (I, \Gamma)\) is a level, then one can define the group
\[\eunit(L) = \langle \tau_{i, j}(a), \tau_i(h) \mid a \in e_i I e_j, h \in \Gamma \cdot e_i, 0 \neq i \neq \pm j \neq 0 \rangle\]
and the elementary level group
\[\eunit(P, L) = {^{\eunit(P)}\eunit(L)}.\]

Let us define \(\alpha(g) = (g, \inv{g^{-1}})\) and \(\eps = (e_0, e_+, -e_0)\), where \(e_+ = \sum_{i > 0} e_i\) and \(e_- = \sum_{i < 0} e_i\). It is easy to see that \(\pi(\eps) = e_0\) and \(\tr(\eps) = e_+ - \zeta\), where \(\zeta = (0, 1)\).

\begin{lemma}\label{LevelGroup}
The map \(\alpha \colon C^* \rar \mathcal A^*\) is a homomorphism, \(\inv{\alpha(g)} = \alpha(g^{-1})\), and there are the following identities
\begin{align*}
\alpha(\tau_{i, j}(a)) &= 1 + a - \inv a; \\
\alpha(\tau_i(h)) &= 1 + \tr(h) + \pi(h) - \inv{\pi(h)}; \\
\eps \cdot \alpha(\tau_{i, j}(a)) \dotminus \eps &= \begin{cases}
 \phi(\inv a), &\text{if }i < 0 < j, \\
 \phi(a), &\text{if }i > 0 > j, \\
 \dot 0, &\text{if }ij > 0; \end{cases} \\
\eps \cdot \alpha(\tau_i(h)) \dotminus \eps &= \begin{cases}
 \dotminus h \cdot (-1) \dotminus \phi(\pi(h)), &\text{if }i > 0, \\
 h, &\text{if }i < 0. \end{cases}
\end{align*}
\end{lemma}
\begin{proof}
All claims are followed from the definitions.
\end{proof}

Now let \(L\) be an augmented level. The principal level group of augmented level \(L\) is the group
\[\unit(P, L) = \{g \in C^* \mid \alpha(g) - 1 \in I, \eps \cdot \alpha(g) \dotminus \eps \in \Gamma\}.\]
This is a group by lemma \ref{LevelGroup}: indeed, \(\alpha(g^{-1}) - 1 = \inv{(\alpha(g) - 1)}\), \(\alpha(gg') - 1 = (\alpha(g) - 1) \alpha(g') + (\alpha(g') - 1)\), \(\eps \cdot \alpha(g^{-1}) \dotminus \eps = \dotminus (\eps \cdot \alpha(g) \dotminus \eps) \cdot \alpha(g^{-1})\), and \(\eps \cdot \alpha(gg') \dotminus \eps = (\eps \cdot \alpha(g) \dotminus \eps) \cdot \alpha(g') \dotplus (\eps \cdot \alpha(g') \dotminus \eps)\). Moreover, the lemma implies that \(\unit(P, L)\) is normalized by \(\eunit(P)\) (because \({^{\alpha(h)} \alpha(g)} - 1 = {^{\alpha(h)} (\alpha(g) - 1)}\) and \((\eps \cdot {^{\alpha(h)} \alpha(g)} \dotminus \eps) \cdot \alpha(h) = (\eps \cdot \alpha(h) \dotminus \eps) \cdot \alpha(g) \dotplus (\eps \cdot \alpha(g) \dotminus \eps) \dotminus (\eps \cdot \alpha(h) \dotminus \eps)\)) and its level is \(L\). Therefore, the level of \(\eunit(P, L)\) is also \(L\). We also will use the general level group
\[\gunit(P, L) = \{g \in \Norm_{\gl(P)}(\unit(P, L)) \mid [g, \eunit(P, \widehat L)] \subseteq \unit(P, L)\}.\]
Clearly, it is a group normalized by \(\eunit(P)\).

\begin{lemma}\label{GeneralLevelGroup}
Let \(l \geq 3\) and \(L\) be an augmented level. The group \(\gunit(P, L)\) contains \(\unit(P, L)\) and its level is \(L\). Besides, \(\gunit(P, L)\) is contained in the set \(\gunit'(P, L) = \{g \in \gl(P) \mid [\{g, g^{-1}\}, \eunit(P, \widehat L)] \subseteq \unit(P, L)\}\), which can also be defined through the equations
\begin{align*}
{^{\alpha(g^{\pm 1})} a} - a \in I & \text{ for all } a \in \lfloor \widehat I \rfloor + K, \\
\eps \cdot {^{\alpha(g^{\pm 1})} a} \dotminus \eps \cdot a\alpha(g^{\mp 1}) \in \Gamma & \text{ for all } a \in (1 - e_0) \widehat I, \\
\eps \cdot {^{\alpha(g^{\pm 1})}\pi(h)} \dotminus \eps \cdot \pi(h) \alpha(g^{\mp 1}) \dotplus h \cdot \alpha(g^{\mp 1}) \dotminus h \in \Gamma & \text{ for all } h \in \lfloor \widehat \Gamma \rfloor.
\end{align*}
Moreover, \([\gunit'(P, \lfloor L \rfloor), \unit(P, \lfloor \widehat I \rfloor + K, \lfloor \widehat \Gamma \rfloor)] \subseteq \unit(P, \lfloor L \rfloor)\) and \(\gunit(P, \lfloor L \rfloor) = \gunit'(P, \lfloor L \rfloor)\).
\end{lemma}
\begin{proof}
Suppose that \(g \in C^*\) satisfies \([g, \eunit(P, \widehat L)] \subseteq \unit(P, L)\). This is equivalent to \(\alpha([g, \tau_{i, j}(a)]) - 1 \in I\), \(\eps \cdot \alpha([g, \tau_{i, j}(a)]) \dotminus \eps \in \Gamma\) for all \(a \in e_i \widehat I e_j\) (where \(0 \neq i \neq \pm j \neq 0\)) and \(\alpha([g, \tau_i(h)]) - 1 \in I\), \(\eps \cdot \alpha([g, \tau_i(h)]) \dotminus \eps \in \Gamma\) for all \(h \in \widehat \Gamma \cdot e_i\) (where \(i \neq 0\)). Then \({^{\alpha(g)} (a - \inv a)} \equiv a - \inv a \mod I\) and \({^{\alpha(g)} (\tr(h) + \pi(h) - \inv{\pi(h)})} \equiv \tr(h) + \pi(h) - \inv{\pi(h)} \mod I\). After multiplying the congruences of the first kind for various \(i\) and \(j\) (we use only that \(l \geq 2\)) we obtain \({^{\alpha(g)} a} \equiv a \mod I\) for all \(a \in (1 - e_0) \widehat I (1 - e_0)\). Then, using the congruences of the second kind, we obtain the first equation on \(g\), that is \({^{\alpha(g)} a} \equiv a \mod I\) for all \(a \in \lfloor \widehat I \rfloor + K\).

Further, we know that \(\eps \cdot {^{\alpha(g)} (1 + a - \inv a)} \equiv \eps \cdot (1 + a - \inv a) \mod \Gamma\) for all \(a \in e_i \widehat I e_j\) (where \(0 \neq i \neq \pm j \neq 0\)). It follows that \(\eps \cdot {^{\alpha(g)} a} \in \Gamma\) (note that \(\Gamma \leqt \langle \eps \cdot (I + \lfloor \widehat I \rfloor + K) \rangle \dotplus \Gamma\) and \(\phi(\zeta (a - \inv a)) = \phi(a)\) for all \(a \in \mathcal A\)). Therefore, \(\eps \cdot {^{\alpha(g)} a} \equiv \eps \cdot a \alpha(g^{-1}) \mod \Gamma\) for all \(a \in (1 - e_0) \widehat I\). Moreover, \(\eps \cdot {^{\alpha(g)}(1 + \tr(h) + \pi(h) - \inv{\pi(h)})} \equiv \eps \cdot (1 + \tr(h) + \pi(h) - \inv{\pi(h)}) \mod \Gamma\) for all \(h \in \widehat \Gamma \cdot e_i\) and \(i \neq 0\). After simplifications we obtain \(\eps \cdot {^{\alpha(g)}\pi(h)} \dotplus \eps \cdot {^{\alpha(g)}(-\inv{\pi(h)})} \equiv \dotminus h \cdot (-1) \dotplus \phi(\zeta {^{\alpha(g)} \tr(h)}) \mod \Gamma\). Recall that \(\eps \cdot {^{\alpha(g)} (-\inv{\pi(h)})} \equiv \eps \cdot (-\inv{\pi(h)} \alpha(g^{-1})) = 0\), hence \(\eps \cdot {^{\alpha(g)} \pi(h)} \dotminus \eps \cdot \pi(h) \alpha(g^{-1}) \dotminus h^* \cdot (-\alpha(g^{-1})) \dotplus h \cdot (-1) \in \Gamma\). After changing \(h\) into \(\dotminus h \cdot (-1)\) the third equation on \(g\) follows.

The level of \(\gunit(P, L)\) can be found from the equations (it is sufficient to use the first and the second ones). If we prove \([\gunit'(P, \lfloor L \rfloor), \unit(P, \lfloor \widehat I \rfloor + K, \lfloor \widehat \Gamma \rfloor)] \subseteq \unit(P, \lfloor L \rfloor)\), then it will be clear that \(\gunit'(P, \lfloor L \rfloor)\) is a group and is equal to \(\gunit(P, \lfloor L \rfloor)\). Let \([g, \eunit(P, \widehat L)] \subseteq \unit(P, \lfloor L \rfloor)\) and \(x \in \eunit(P, \lfloor \widehat I \rfloor + K, \lfloor \widehat \Gamma \rfloor)\), then
\[\alpha([g, x]) - 1 = ({^{\alpha(g)} \alpha(x)} - \alpha(x)) \alpha(x^{-1}) \in \lfloor I \rfloor\]
and if \(h = \eps \cdot \alpha(x) \dotminus \eps \in \lfloor \widehat \Gamma \rfloor\), then
\begin{align*}
(\eps \cdot \alpha([g, x]) &\dotminus \eps) \cdot \alpha(x) \equiv\\
&\equiv \eps \cdot {^{\alpha(g)}(e_0 \alpha(x))} \dotplus \eps \cdot {^{\alpha(g)}((1 - e_0)\alpha(x))} \dotminus \eps \cdot \alpha(x) \equiv\\
&\equiv \eps \cdot {^{\alpha(g)}\pi(\eps \cdot \alpha(x) \dotminus \eps)} \dotminus \phi({^{\alpha(g)}(\zeta e_0(\alpha(x) - 1))}) \dotplus \eps \cdot {^{\alpha(g)}e_0} \dotplus \phantom{a}\\
&\quad\dotplus\eps \cdot (1 - e_0)\alpha(x)\alpha(g^{-1}) \dotminus \eps \cdot \alpha(x) \equiv\\
&\equiv \eps \cdot \alpha(x) \dotminus \eps \dotminus (\eps \cdot \alpha(x) \dotminus \eps) \cdot \alpha(g^{-1}) \dotplus \eps \cdot e_0(\alpha(x) - 1)\alpha(g^{-1}) \dotminus \phantom{a}\\
&\quad\dotminus \phi({^{\alpha(g)}(\zeta e_0(\alpha(x) - 1))}) \dotplus \eps \dotminus \eps \cdot \alpha(x) \dotminus \phantom{a}\\
&\quad\dotminus \eps \cdot (1 - e_0) \alpha(g^{-1}) \dotplus \eps \cdot (1 - e_0)\alpha(x)\alpha(g^{-1}) \equiv \dot 0 \mod \Gamma.\qedhere
\end{align*}
\end{proof}

\section{Localization and roots}

Let \(S \leq \herm(K)^\bullet\) be a multiplicative subset, \(L\) be a level. Then \(S^{-1} L \subseteq S^{-1} \mathcal A \times S^{-1} \mathcal H\) can be considered as a level (if instead of \(\Lambda\) we will take \(S^{-1} \Lambda\)), hence one can define the corresponding subgroups \(\eunit(S^{-1} P, S^{-1} L)\), \(\unit(S^{-1} P, S^{-1} L)\) and \(\gunit(S^{-1} P, S^{-1} L)\) in \((S^{-1} C)^*\) (in general, they are not related to the bimodule \(S^{-1} P\), because its endomorphism ring can be strictly bigger than \(S^{-1} C\)). Let \(L_0\) be the level of \(\eunit(P)\), \(\eunit(S^{-1} P)\) be the group \(\eunit(S^{-1} P, S^{-1} L_0)\), then \(\eunit(S^{-1} P, S^{-1} L) = {^{\eunit(S^{-1} P)} \eunit(S^{-1} L)}\). The localization homomorphism will be denoted by \(\Psi_S\).

If \(L = (I, \Gamma)\) is a level and \(s \in S\), then let \(L \cdot s = (Is, \Gamma \cdot s + \phi(Is))\), it is also a level. Let us introduce the system of subgroups \(\Omega_S(P, L) = \{\eunit(P, L \cdot s) \mid s \in S\}\) in \(\eunit(P, L)\). We a going to prove that \(\Omega_S(P, L)\) is a neighborhood base of the identity for some group topology (it will be called a base for short) on \(\eunit(P, L)\), not necessary Hausdorff, and that \(\eunit(P, \widehat L)\) acts continuously by conjugation on \(\eunit(P, L)\). Moreover, we a going to prove similar statements for \(\eunit(S^{-1} P, S^{-1} L)\) with the subgroup system \(\Psi_S(\Omega_S(P, L))\).

Let \(\Phi = \{\pm \mathrm e_i \pm \mathrm e_j, \pm \mathrm e_k, \pm 2\mathrm e_k \mid 1 \leq i < j \leq l; 1 \leq k \leq l\} \subseteq \mathbb R^l\), it is a non-reduced (crystallographic) root system of type \(BC_l\). Its elements (i.\,e. roots) of length \(2\) are called long, of length \(\sqrt 2\) --- the short ones, and of length \(1\) --- the ultrashort ones. If \(\alpha \in \Phi\) is a root and \(L = (I, \Gamma)\) is a level, then one can defines a subgroup \(U_\alpha(L) \leq \eunit(P, L)\) in the following way (where \(\mathrm e_i = -\mathrm e_{-i}\) for \(i < 0\)):
\[
U_\alpha(L) = \begin{cases}
\tau_{i, j}(I), &\alpha = \mathrm e_i - \mathrm e_j, 0 \neq i \neq \pm j \neq 0;\\
\tau_i(\Gamma), &\alpha = -\mathrm e_i, 0 \neq i;\\
\tau_i(\phi(I)), &\alpha = -2\mathrm e_i, 0 \neq i.
\end{cases}
\]

Clearly, \(U_{2\mathrm e_i}(L) \leqt U_{\mathrm e_i}(L)\). From lemma \ref{LevelTransvections} it follows that
\[[U_\alpha(L), U_\beta(\widehat L)] \leq \prod_{\substack{i \alpha + j \beta \in \Phi\\ i, j > 0}} U_{i\alpha + j\beta}(L),\]
if \(i \alpha + j \beta \neq 0\) for all \(i, j > 0\) (the lemma also implies that the product in the right hand side is a group and is independent from the order of factors).

\[\xymatrix@R=6.5pt@C=7.5pt@!0{
&& \gamma &&&& \alpha &&\\
&&&&&&&&\\
&&&&&&&&\\
&&&&&&&&\\
-\beta &&&&&&&& \beta \\
&&&& *{} \ar[uuuurr]\ar[uuuull]\ar[ddddrr]\ar[ddddll]\ar[rrrr]\ar[llll] &&&&\\
&&&&&&&&\\
&&&&&&&&\\
&&&&&&&&\\
&&&&&&&&\\
&& -\alpha &&&& -\gamma &&\\
&&&&&&&&\\
&&&&&&&&\\
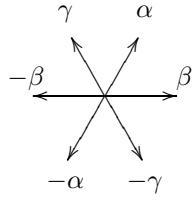
&&&&\text{Root system of type \(A_2\)}&&&&\\
}
\qquad
\xymatrix@R=6.5pt@C=6.5pt@!0{
2\gamma &&&& \alpha + \gamma &&&& 2\alpha \\
&&&&&&&&\\
&& \,\,\gamma &&&& \alpha\,\, &&\\
&& *{} \ar[uull] &&&& *{} \ar[uurr] &&\\
-\beta &&&&&&&& \beta\\
&&&& *{} \ar[uuuu]\ar[rrrr]\ar[llll]\ar[dddd]\ar[uurr]\ar[uull]\ar[ddrr]\ar[ddll] &&&&\\
&&&&&&&&\\
&& *{} \ar[ddll] &&&& *{} \ar[ddrr] &&\\
&& \,\ -\alpha &&&& -\gamma\ \,\, &&\\
&&&&&&&&\\
-2\alpha\,\,\,\,\,\, &&&& -\alpha - \gamma &&&& \,\,\,\,\,\,-2\gamma \\
&&&&&&&&\\
&&&&&&&&\\
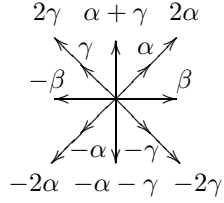
&&&&\text{Root system of type \(BC_2\)}&&&&
}\]

\begin{lemma}\label{LevelPerfection}
If \(l \geq 3\) and \(L\) is a level, then \([\eunit(P, L), \eunit(P)] = \eunit(P, L)\).
\end{lemma}
\begin{proof}
It is sufficient to prove that \(\eunit(L) \leq [\eunit(L), \eunit(L_0)]\). If \(\alpha\) is a short root, then it can be presented as \(\alpha = \beta + \gamma\), where \(\beta\) and \(\gamma\) are also short, hence \(U_\alpha(L) = [U_\beta(L), U_\gamma(L_0))]\). If \(\alpha\) is ultrashort, then \(\alpha = \beta + \gamma\), where \(\beta\) is short and \(\gamma\) is ultrashort, hence
\begin{align*}
U_\alpha(L) &\leq \langle U_{2\alpha}(L), [U_\gamma(L), U_\beta(L_0)], U_{\alpha + \gamma}(L) \rangle \leq\\
&\leq \langle [U_{\alpha + \gamma}(L), U_\beta(L_0))], [U_\gamma(L), U_\beta(L_0)], U_{\alpha + \gamma}(L) \rangle.\qedhere
\end{align*}
\end{proof}

Note that if \(H, H' \in \Omega_S(P, L)\), then there exists \(H'' \in \Omega_S(P, L)\) such that \(H'' \leq H \cap H'\) (if \(H = \eunit(P, L \cdot s)\) and \(H' = \eunit(P, L \cdot s')\), then one can take \(H'' = \eunit(P, L \cdot ss')\)). Therefore, in order to prove that \(\Omega_S(P, L)\) is a base and \(\eunit(P, \widehat L)\) acts continuously on \(\eunit(P, L)\), is is sufficient to prove continuity in \(1\) (or in \((1, 1)\)) of maps \(\eunit(P, L) \rar \eunit(P, L), g \mapsto [g, x]\) for \(x \in \eunit(P, \widehat L)\) fixed, \(\eunit(P, \widehat L) \rar \eunit(P, L), x \mapsto [g, x]\) for \(g \in \eunit(P, \widehat L)\) fixed, and \(\eunit(P, L) \times \eunit(P, \widehat L) \rar \eunit(P, L), (g, x) \mapsto [g, x]\).

\begin{lemma}\label{GeneratorsOfElementaryLevelGroup}
If \(l \geq 3\) and \(L\) is a level, then \(\eunit(P, L) = \langle {^{U_{-\alpha}(L_0)} U_\alpha(L)} \mid \alpha \in \Phi \rangle\). If \(S \leq \herm(K)^\bullet\) is a multiplicative subset, then \(\Omega_S(P, L)\) is a base of \(\eunit(P, L)\) and \(\Psi_S(\Omega_S(P, L))\) is a base of \(\eunit(S^{-1} P, S^{-1} L)\). Moreover, \(\eunit(P, \widehat L)\) acts continuously by conjugation on \(\eunit(P, L)\) and \(\eunit(S^{-1} P, S^{-1} \widehat L)\) acts continuously by conjugation on \(\eunit(S^{-1} P, S^{-1} L)\).
\end{lemma}
\begin{proof}
Let \(\eunit'(L) = \langle {^{U_{-\alpha}(L_0)} U_\alpha(L)} \mid \alpha \in \Phi \rangle\), then \(\eunit(L) \leq \eunit'(L) \leq \eunit(P, L)\). At first we will prove that \(\eunit'(L)\) is closed under conjugation by elements \(x \in U_\delta(\widehat L)\) and these conjugations are continuous at \(1 \in \eunit'(L) = \eunit(P, L)\). It is sufficient to prove the claim for subgroups \({^{U_{-\alpha}(L_0)} U_\alpha(L)}\) (with the base \(\{{^{U_{-\alpha}(L_0)} U_\alpha(L \cdot s)} \mid s \in S\}\)), which are mapped into \(\eunit(P, L)\) by conjugation.

If \(\alpha\) and \(\delta\) are linearly independent, then
\begin{align*}
[U_\delta(\widehat L), {^{U_{-\alpha}(L_0)} U_\alpha(L \cdot s)}] &\leq \langle \eunit'(L \cdot s), {^{U_{-\alpha}(L_0)} [{^{U_{-\alpha}(L_0)} U_\delta(\widehat L)}, U_\alpha(L \cdot s)]} \rangle \leq\\
&\leq \langle \eunit'(L \cdot s), \prod_{\substack{i \alpha + j \delta \in \Phi\\ j > 0}} U_{i \alpha + j \beta}(L \cdot s) \rangle \leq \eunit'(L \cdot s).
\end{align*}

If \(\alpha = \pm \delta\) is short, then \(\alpha = \beta + \gamma\), where \(\beta\) and \(\gamma\) are also short, hence
\begin{align*}
{^{U_\alpha(\widehat L) U_{-\alpha}(\widehat L)} U_\alpha(L \cdot s^2)} &\leq [{^{U_\alpha(\widehat L) U_{-\alpha}(\widehat L)} U_\beta(L \cdot s)}, {^{U_\alpha(\widehat L) U_{-\alpha}(\widehat L)} U_\gamma(L_0 \cdot s)}] \leq\\
&\leq [U_\beta(L \cdot s) U_{-\gamma}(L \cdot s), U_\gamma(L \cdot s) U_{-\beta}(L \cdot s) U_\gamma(L_0) U_{-\beta}(L_0)] \leq\\
&\leq \eunit'(L \cdot s).
\end{align*}

If \(\alpha = \pm\delta\) is ultrashort, then \(\alpha = \beta + \gamma\), where \(\beta\) is short and \(\gamma\) is ultrashort, hence
\begin{align*}
{^{U_\alpha(\widehat L) U_{-\alpha}(\widehat L)} U_{2\alpha}(L \cdot s^2)} &\leq [{^{U_\alpha(\widehat L) U_{-\alpha}(\widehat L)} U_\beta(L_0 \cdot s)}, {^{U_\alpha(\widehat L) U_{-\alpha}(\widehat L)} U_{\alpha + \gamma}(L \cdot s)}] \leq\\
&\leq [U_\beta(L_0) U_{-\alpha - \gamma}(L_0) U_{-\gamma}(L_0) U_\beta(L \cdot s) U_{-\alpha - \gamma}(L \cdot s) U_{-\gamma}(L \cdot s),\\
&\qquad U_{\alpha + \gamma}(L \cdot s) U_{-\beta}(L \cdot s) U_{\gamma}(L \cdot s)] \leq \eunit'(L \cdot s)
\end{align*}
and
\begin{align*}
{^{U_{\alpha}(\widehat L) U_{-\alpha}(\widehat L)} U_\alpha(L \cdot s^2)} &= \langle \eunit'(L \cdot s), [{^{U_\alpha(\widehat L) U_{-\alpha}(\widehat L)}U_\gamma(L \cdot s)}, {^{U_\alpha(\widehat L) U_{-\alpha}(\widehat L)} U_\beta(L_0 \cdot s)}] \rangle \leq\\
&\leq \langle \eunit'(L \cdot s), [U_{\alpha + \gamma}(L \cdot s) U_{-\beta}(L \cdot s) U_\gamma(L \cdot s),\\
&\qquad U_\beta(L_0) U_{-\alpha - \gamma}(L_0) U_{-\gamma}(L_0) U_\beta(L \cdot s) U_{-\alpha - \gamma}(L \cdot s) U_{-\gamma}(L \cdot s) \rangle \leq\\
&\leq \eunit'(L \cdot s).
\end{align*}

It can be proved similarly that \({^x (-)} \colon \eunit(S^{-1} P, S^{-1} L) \rar \eunit(S^{-1} P, S^{-1} L)\) is continuous at \(1\) for \(x \in \eunit(S^{-1} P, S^{-1} \widehat L)\) and that both maps \((g, x) \mapsto [g, x]\) are continuous at \((1, 1)\). It remains to prove continuity of the map \(x \mapsto [g, x]\), where \(g \in \eunit(P, L)\) (and similarly for \(\eunit(S^{-1} P, S^{-1} L)\)). Without loss of generality, \(g \in U_\alpha(L)\), then \([U_\delta(\widehat L \cdot s), U_\alpha(L)] \leq \prod_{\substack{i \alpha + j \delta \in \Phi\\ i, j > 0}} U_{i \alpha + j \delta}(L \cdot s)\) for linearly independent or codirectional \(\alpha\) and \(\delta\). If \(\alpha = -\beta\) is short and \(\alpha = \beta + \gamma\), where both \(\beta\) and \(\gamma\) are short, then
\begin{align*}
[U_{-\alpha}(\widehat L \cdot s^2), U_\alpha(L)] = [\langle U_{-\beta}(\widehat L \cdot s) U_{-\gamma}(\widehat L \cdot s) \rangle, U_\alpha(L)] \leq \eunit'(L \cdot s),
\end{align*}
and similarly in the case when \(\alpha = -\beta\) is ultrashort.
\end{proof}

Note that the embeddings \(\eunit(P, L) \rar \eunit(P, L')\) and \(\eunit(S^{-1} P, S^{-1} L) \rar \eunit(S^{-1} P, S^{-1} L')\) are continuous, if the level \(L\) is contained in \(L'\) (that is \(\lfloor I \rfloor \leq \lfloor I' \rfloor\) and \(\lfloor \Gamma \rfloor \leq \lfloor \Gamma' \rfloor\)).

\begin{lemma}\label{BoundaryGenerators}
Let \(l \geq 3\) and \(L\) be a level. Then the group \(\eunit(P, L)\) is generated by the subgroups \(U_{\pm \mathrm e_l \pm \mathrm e_i}(L)\) for \(1 - l \leq i \leq l - 1\) as a subgroup of \(C^* = \gl(P)\) normalized by \(\eunit(P, \widehat L)\) (and \(\eunit(P, \widehat L)\) itself is generated by similar subgroups as an abstract group).
\end{lemma}
\begin{proof}
This follows trivially from lemma \ref{LevelTransvections}.
\end{proof}

\section{General level groups}

Recall that a quasi-finite \(K\)-algebra is a direct limit of finite \(K\)-algebras (one can consider it as a noncommutative integrality). This condition on an algebra has equivalent reformulations (which are written in \cite{NonabelianBak}), and these reformulations imply that a subalgebra \(B\) of a quasi-finite algebra \(A\) is quasi-finite itself (we also will use the equality \(B^* = A^* \cap B\)). It is also useful to note that quasi-finiteness over \(K\) is equivalent to quasi-finiteness over \(\herm(K)\), since \(K\) is integral over \(\herm(K)\) (the integral dependence equation for \(k \in K\) is \(k^2 - k(k + \inv k) + (k \inv k) = 0\)). Finally, we will use that in a quasi-finite algebra one-sided invertibility is equivalent to a two-sided one.

We a going to prove the inclusion \([\gunit(P, L), \eunit(P, \widehat L)] \subseteq \eunit(P, L)\) for all augmented levels \(L\) if \(C\) is quasi-finite over \(K\) under certain additional assumption. First of all we will prove one lemma for \(K\) local.

Recall that \(\gunit'(P, L) = \{g \in \gl(P) \mid [\{g, g^{-1}\}, \eunit(P, \widehat L)] \subseteq \unit(P, L)\}\) for an augmented level \(L\). Clearly, this set is closed under conjugation by \(\eunit(P, \widehat L)\) and under multiplication by \(\eunit(P, L)\) from both sides. Let \(\gunit'_{\widehat k}(P, L) = \{g \in \gunit'(P, L) \mid e_k \alpha(g) = e_k \alpha(g) e_k = \alpha(g) e_k\}\) for all \(k \neq 0\). Obviously, \([\gunit'_{\widehat k}(P, L), \eunit(P, \widehat L)] \subseteq \eunit(P, L)\) by lemma \ref{BoundaryGenerators}.

\begin{lemma}\label{LocalLevi}
Suppose that \(l \geq 3\), \(K\) is a local ring with the maximal ideal \(\mathfrak m\), \(C\) is a finite \(K\)-algebra, \({^{\alpha(g^{\pm 1})} a} \equiv a \mod I\) and \(\eps \cdot {^{\alpha(g^{\pm 1})} a} \in \Gamma\) for all \(a \in e_{\pm k} \widehat I e_{\mp k}\), and also \(e_k \alpha(g) e_k\) is invertible in \(e_k \mathcal A e_k\). Then there exist \(h, h' \in \eunit(P, L)\) such that \(e_k \alpha(hgh') = e_k \alpha(hgh') e_k = \alpha(hgh') e_k\).
\end{lemma}
\begin{proof}
Note that always \(e_i \alpha(g) e_k \widehat I e_{-k} \alpha(g^{-1}) e_{-k} \widehat I e_k \leq I\) and \(e_k \in e_k \widehat I e_{-k} \alpha(g^{-1}) e_{-k} \widehat I e_k \alpha(g) e_k\) for all \(0 \neq i \neq \pm k\) (since \(e_{-k} \alpha(g^{-1}) e_{-k}\) is invertible in \(e_{-k} \mathcal A e_{-k}\)). If \(e_k = x \alpha(g) e_k\) for \(x \in e_k \widehat I e_{-k} \alpha(g^{-1}) e_{-k} \widehat I e_k\), then \(e_i \alpha(\tau_{i, k}(-e_i \alpha(g) x) g) e_k = 0\). Since multiplication by such a transvection preserves the condition, we may assume that \(e_i \alpha(g) e_k = 0\). Now let \(h = \eps \cdot \alpha(g) x \in \Gamma\), then \((1 - e_k)\alpha(\tau_k(h \cdot (-1)) g)e_k = 0\) and multiplication by this transvection also preserves the condition on \(g\).

Applying the same transformations for \(g^{-1}\) instead of \(g\) (and \(-k\) instead of \(k\)) we can obtain the equality \(e_k \alpha(g) = e_k \alpha(g) e_k\).
\end{proof}

\begin{prop}\label{LocalSplitting}
Suppose that \(l \geq 3\), \(K\) is a local ring with the maximal ideal \(\mathfrak m\), \(C\) is a finite \(K\)-algebra and \(e_0 C e_0\) is generated by less than \(4l^2\) elements as a \(K\)-module. Then \(\gunit'(P, L) = \eunit(P, L) \gunit'_{\widehat k}(P, L)\) for all \(k \neq 0\).
\end{prop}
\begin{proof}
Without loss of generality, \(k = l\). We will prove that every \(g \in \gunit'(P, L)\) is in the set \(\eunit(P, L) \gunit'_{\widehat l}(P, L)\). We will conjugate it by elements from \(\eunit(P, \widehat L)\) and multiply it from both sides by elements from \(\eunit(P, L)\), until we obtain an element from \(\gunit'_{\widehat l}(P, L)\). For convenience we will work in \(\mathcal A\), applying conjugations and multiplications directly to \(\alpha(g)\).

At first let us prove that there exist \(p \in e_l \widehat I\) and \(q \in \widehat I e_l\) such that \(p \, {^{\alpha(g)} e_l} \, q\) in invertible in \(e_l \widehat I e_l\). In order to do it we note that \({^{\alpha(g)} ((1 - e_0) \widehat I (1 - e_0))} \subseteq (1 - e_0) \widehat I (1 - e_0) + I\). Without loss of generality, \(K\) is a field and \(\mathcal A\) is semi-simple (since we can take the factor by the Jacobson radical \(\mathrm J(\mathcal A)\)), then \(\mathcal A\) is a finite-dimensional \(K\)-algebra. Let \(A = \widehat I (1 - e_0) \widehat I + K\) and \(f_0, f_1, \ldots, f_m\) be liftings to \(A\) of the primitive central idempotents of \(A / \mathrm J(A)\) (such that \(f_i\) form a complete family of orthogonal idempotents), one can assume that \(f_i \in \widehat I (1 - e_0) \widehat I\) for all \(i \neq 0\), \(f_0 \widehat I (1 - e_0) \widehat I \leq \mathrm J(A)\) and \(f_0 = e_0 f_0 e_0\) (if \(A = \widehat I (1 - e_0) \widehat I\), then let \(f_0 = 0\)). Let also \(B = I + \widehat I (1 - e_0) \widehat I + K\). It is easy to show that the classes of \(f_i\) in \(B / \mathrm J(B)\) are primitive central idempotents for all \(i \neq 0\). Moreover, \((1 - e_0) A (1 - e_0) = (1 - e_0) B (1 - e_0)\) is semi-simple.

Let us fix an index \(1 \leq i \leq m\). Clearly, \(X = {^{\alpha(g)} ((1 - e_0) f_i A (1 - e_0))} \subseteq B\) is a semi-simple subalgebra (generally their identity elements are different) and \(f_j X \subseteq f_j A[\zeta] f_i \mathcal A = 0\) for all \(1 \leq j \leq m\) and \(j \neq i\) (since \(f_j \mathcal A (1 - e_0) f_i \subseteq A[\zeta]\)). Moreover, \(f_0 X \subseteq \mathrm J(B)\), because the image of \(f_0 X\) in \(B / \mathrm J(B)\) is a semi-simple subalgebra contained in \(e_0 B e_0 / \mathrm J(e_0 B e_0)\), all simple factors of \(X\) have dimension at least \(4l^2\), and \(e_0 B e_0 / \mathrm J(e_0 B e_0)\) can be embedded in the product \(e_0 B \zeta e_0 / \mathrm J (e_0 B \zeta e_0) \times e_0 B(1 - \zeta) e_0 / \mathrm J (e_0 B (1 - \zeta) e_0)\), where both factors have dimension strictly less than \(4l^2\). It follows that \(X \subseteq f_i X + \mathrm J(B) \subseteq \widehat I f_i \widehat I + \mathrm J(B)\), hence there are \(p \in e_l \widehat I\) and \(q \in \widehat I e_l\) such that \(p \, {^{\alpha(g)} e_l} \, q\) is invertible in \(e_l \widehat I e_l\).

Now let us show that \(p\) can be taken of type \(e_l\alpha(h)\), where \(h \in \eunit(P, \widehat L)\) and \(e_{-l} \alpha(h) = e_{-l}\). We will work in the algebra \(T = B / \mathrm J(B)\). Let \(p' = e_l \alpha(h)\) be such that \(h \in \eunit(P, \widehat L)\), \(e_{-l} \alpha(h) = e_{-l}\) and the ranks of \(p' \, {^{\alpha(g)} e_l} \, q\) are maximal possible (in every simple factor of \(T\)). Let us denote \(\alpha(hg) e_l \alpha(g^{-1}) q\) as \(f\). It is easy to see that \(e_i f e_l \in T e_l f e_l\) for all \(0 \neq i \neq \pm l\), otherwise we can multiply \(f\) by some element of type \(\tau_{l, i}(a)\) from the left, increasing some rank of \(e_l f\). Next, \(e_{-l} f e_l \in T e_l f e_l\), since otherwise we can multiply \(f\) by \(\tau_{l, 2 - l}(a) \tau_{l - 1, l - 2}(b) \tau_{l, l - 2}(c)\) from the left, increasing some rank of \(e_+ f e_l\). Finally, \(e_0 f e_l \in T e_l f e_l\), because otherwise we can multiply \(f\) by an element of type \(\tau_{-l}(h)\) from the left, increasing some rank of \((1 - e_0) f e_l\). Hence \(f e_l \in T e_l f e_l\), then \(e_l f e_l\) has the same ranks as \(e_l\), that is this element is invertible in \(e_l T e_l\).

Finally, lemma \ref{LocalLevi} can be applied to \({^h g}\).
\end{proof}

\begin{theorem}\label{CommutantFormula}
Suppose that \(l \geq 3\), \(C\) is a quasi-finite \(K\)-algebra and \(e_0 C e_0\) is generated by less than \(4l^2\) elements as a \(K\)-module. Then for any augmented level \(L\) there is the inclusion \([\gunit'(P, L), \eunit(P, \widehat L)] \subseteq \eunit(P, L)\).
\end{theorem}
\begin{proof}
Fix \(g \in \gunit'(P, L)\), then without loss of generality \(\herm(K)\) is Noetherian (for examply it is finitely generated over \(\mathbb Z\)), \(K\) and \(C\) are finite \(\herm(K)\)-algebras. Consider the ideal
\[\mathfrak a = \{k \in \herm(K) \mid [g, \eunit(P, \widehat L \cdot k)] \subseteq \eunit(P, L)\} \leqt K.\]
It is trivial to check that it is an ideal. Let \(\mathfrak m \leqt \herm(K)\) be a maximal ideal, then \(K_{\mathfrak m}\) is local and proposition \ref{LocalSplitting} can be applied to \(C_{\mathfrak m}\) as a \(K_{\mathfrak m}\)-algebra. An element \(\Psi_{\mathfrak m}(g)\) is a product of elements from \(\eunit(P_{\mathfrak m}, L_{\mathfrak m})\) and \(\diag(P_{\mathfrak m}, L_{\mathfrak m})\), hence for all \(H \in \Omega_{\mathfrak m}(P, L)\) there exists \(H' \in \Omega_{\mathfrak m}(P, \widehat L)\) such that \(\Psi_{\mathfrak m}([g, H']) \subseteq \Psi_{\mathfrak m}(H)\). We can take \(H = \eunit(P, L \cdot s)\) and \(H' = \eunit(P, \widehat L \cdot ss')\) for such \(s\) that \(\Psi_{\mathfrak m}|_{Cs}\) is injective, then \([g, H'] \subseteq H\) and \(ss' \in \mathfrak a \not \leq \mathfrak m\). Therefore, \(\mathfrak a = (1)\).
\end{proof}

Therefore, \(\gunit(P, L) \leq \gunit'(P, L) = \{g \in \gl(P) \mid [\{g, g^{-1}\}, \eunit(P, \widehat L)] \subseteq \eunit(P, L)\} = \gunit(P, \lfloor L \rfloor)\) under assumptions of the theorem.

Let us give an example that shows necessity of the condition on \(e_0 C e_0\). Let \(S = R = K\) with the trivial involution, \(2 \in K^*\), \(\lambda_S = \lambda_R = 1\), \(P_i\) are free \(K\)-modules of rank \(1\) for all \(i \neq 0\) and \(P_0 \cong \mathrm H(P^l)\) as a hermitian space. Let also \(A_S = A_K = \Heis(K)\) and \(A_R = \Heis(B) / \sum^\cdot_{i \neq 0} q(P_i)\) (in this case both Heisenberg groups are abelian). Then \(\unit(P) = \{g \in \gl(P) \mid g^{\mathrm T} b = b g^{-1}, g_{0, i} = 0\text{ for } i \neq 0\}\), where \(g^{\mathrm T}\) is the transpose, \(b\) is the matrix of the split symmetric bilinear form, and \(g_{0, i}\) is \(e_0 g e_i\) in out previous notation. Since \(g \in \unit(P)\) is an orthogonal operator, it follows that \(g_{i, 0} = 0\). In other words, \(\unit(P) = \mathrm O(2l, K) \times \mathrm O(P_0)\), where the first factor means the split orthogonal group and \(\eunit(P) = \mathrm{EO}(2l, K)\). Further, \(\unit(P, \lceil L_0 \rceil) = \mathrm O(2l, K) \times \gl(P_0)\). There is some augmented level \(L_0'\) (which corresponds to the level \(L_0\)) such that \(\unit(P, L_0') = \unit(P)\) and \(\gunit(P, L_0') = (\mathrm{GO}(2l, K) \times \mathrm{GO}(P_0)) \rtimes \langle \sigma \rangle\), where \(\sigma\) isometrically swaps \(P_0\) and \(\bigoplus_{i \neq 0} P_i\), \(\sigma^2 = \id\). Clearly, \([\gunit(P, L_0'), \eunit(P)] \not \leq \eunit(P)\). Worse, there does not exists the greatest group with the level \(L_0\), since \(\langle \gunit(P, L_0'), \unit(P, \lceil L_0 \rceil) \rangle = (\gl(2l, P) \times \gl(P_0)) \rtimes \langle \sigma \rangle\) has a strictly greater level.

\section{Global extraction of transvections}

Now we will prove the main theorem. Let \(G \leq \gl(P)\) be a subgroup normalized by \(\eunit(P)\) and \(L\) be its level. We want to prove that \(G \leq \gunit(P, \lfloor L \rfloor)\). More precisely, it will be proved that if \(\eunit(P, L) \leq G\), then either \(G \leq \gunit(P, \lfloor L \rfloor)\), or \(G\) contains an elementary transvection not from \(\eunit(P, L)\).

It is simpler to prove the assertion in the local case, hence we need an auxiliary result that allows us to pass from the local case to the global one. Suppose that \(S \leq \herm(K)^\bullet\) is a multiplicative subset, \(H \in \Psi_S(\Omega_S(P, \widehat L))\), \(H' \in \Psi_S(\Omega_S(P, L))\) (and \(H' \leq H\)), \(g \in G\) and suppose that we found a sequence of elements \(g_0 = \Psi_S(g), g_1, \ldots, g_N\) in \(\gl(S^{-1} P) = (S^{-1} C)^*\) such that \(g_N \notin \eunit(S^{-1} P, S^{-1} L)\) is an elementary transvections (such a transvection will be called nontrivial) and every \(g_{i + 1}\) can be obtained from \(g_i\) by one of the rules
\begin{enumerate} % extraction rules
\item[EX1.] \(g_{i + 1} = f_{i, 0} \, {^{h_i} {g_i}} \, f_{i, 1}\), where \(h_i \in \eunit(S^{-1} P, S^{-1} \widehat L)\) and \(f_{i, 0}, f_{i, 1} \in \eunit(S^{-1}P, S^{-1}L)\);
\item[EX2.] \(g_{i + 1} = f_{i, 0} \, \prod_{k = 1}^{n_i} ([g_i, t_{i, k}]^{\pm 1} \, f_{i, k})\), where \(t_{i, k} \in \eunit(S^{-1} P, S^{-1} \widehat L))\), \(f_{i, k} \in \eunit(S^{-1} P, S^{-1} L))\), \(t_{i, k}^{h'_{i - 1} \ldots h'_0} \in H\) and \([t_{i, k}, b_i]^{h'_{i - 1} \ldots h'_0}, f_{i, k}^{h'_{i - 1} \ldots h'_0} \in H'\) (the elements \(h'_j \in \eunit(S^{-1} P, S^{-1} \widehat L)\) and \(b_i \in \gunit(S^{-1} P, \lfloor S^{-1} L \rfloor)\) will be defined below, their definitions use only the previous \(g_j\));
\item[EX3.] \(g_{i + 1} = g_i x_i\), where \(x_i \in \gunit(S^{-1} P, \lfloor S^{-1} L \rfloor)\);
\item[EX4.] \(g_{i + 1} = g_i^{-1}\), and if before this rule there was an application of the third rule, then somewhere after that application there was an application of the second rule.
\end{enumerate}

Usually we will not use the third rule at all until a certain moment, and after this moment we will not use the fourth one. In all lemmas about local extraction of transvections we will assume that \(g\) is obtained as result of such a sequence and that in this sequence after every application of the third rule there is an application of the second one. In the second rule we will usually use only one factor with \(f_{i, 0} = f_{i, 1} = 1\) and \(t_{i, 1}\) will usually be taken as a transvection from a small enough element of \(\Psi_S(\Omega_S(P, \widehat L))\) (depending of the previous \(g_j\) in order to satisfy the necessary conditions), that is transvections of type \(\tau_{p, q}(a s)\) or \(\tau_p(h \cdot s)\) for \(s \in S\) big enough.

It is easy to prove by induction that \(g_i = a_i b_i\) for some \(a_i\) and \(b_i \in \gunit(S^{-1} P, \lfloor S^{-1} L \rfloor)\). Also \(a_0 = \Psi_S(g)\) and either \(a_{i + 1} = {^{h'_i} a_i^{\pm 1}}\), or \(a_{i + 1} = f'_{i, 0} \, \prod_{k = 1}^{n'_i} ([a_i, t'_{i, k}] \, f'_{i, k})\) (in this case \(h'_i = 1\) and \(b_i = 1\)), where \(h'_i \in \eunit(S^{-1} P, S^{-1} \widehat L)\), \(t'_{i, k} \in \eunit(S^{-1} P, S^{-1} \widehat L)\), \(f'_{i, k} \in \eunit(S^{-1} P, S^{-1} L)\), \({f'_{i, k}}^{h'_{i - 1}, \ldots, h'_0} \in H'\), and \({t'_{i, k}}^{h'_{i - 1} \ldots h'_0} \in H\). In the unique nontrivial case of the second rule one may use the formulas \([a_i b_i, t_{i, k}] = [b_i, t_{i, k}] \, [[t_{i, k}, b_i], a_i] \, [a_i, t_{i, k}]\) and \([t_{i, k}, a_i b_i] = [t_{i, k}, a_i] \, [a_i, [t_{i, k}, b_i]] \, [t_{i, k}, b_i]\). Clearly, these \(a_i\) can be obtained using the same rules (after including additional steps), and for them in the first rule always \(f_{i, 0} = f_{i, 1} = 1\), all \(b_i\) constructed from this new sequence are equal to \(1\) (if \(f_{i, 0} = f_{i, 1} = 1\) in all applications of the first rule, one can take \(h'_i = h_i\)).

Let \(g_N\) be a nontrivial transvection of a root type \(\alpha\). Clearly, there exists a short root \(\beta \in \Phi\) such that the angle between \(\alpha\) and \(\beta\) is obtuse. There also exists \(a \in \widehat I \cdot s\) for \(s \in S\) big enough such that \(g_{N + 1} = [g_N, \tau_\beta(\frac a 1)]\) still is not in \(\eunit(S^{-1} P, S^{-1} L)\), but \(\tau_\beta(\frac a 1)^{h'_N \ldots h'_0} \in H\) and \([b_N, \tau_\beta(\frac a 1)]^{h'_N \ldots h'_0} \in H'\). Clearly, \(g_{N + 1}^{h'_N \ldots h'_0} \in \langle [\Psi_S(g), H], H' \rangle\) (the commutant can be written because we take a commutator at least at the \((N + 1)\)-st step).

In the following lemma a nontrivial transvection is a product of at most two elementary transvections (with an angle \(\frac \pi 4\) between roots) if the product is not in \(\eunit(S^{-1} P, S^{-1} L)\).

\begin{lemma}\label{ContinuityExtraction}
Suppose that \(l \geq 3\), \(\herm(K)\) is Noetherian, \(K\) and \(C\) are finite \(\herm(K)\)-algebras, \(L\) is a level, and \(\Omega^L_S(P) = \{\eunit(P, \langle a \rangle) \mid a \in (1 - e_0) \mathcal A (1 - e_0); \forall s \in S \enskip as \notin I\} \cup \{\eunit(P, \langle h \rangle) \mid h \in \mathcal H \cdot (1 - e_0); \forall s \in S \enskip h \cdot s \notin \Gamma\}\), where \(\langle a \rangle\) and \(\langle h \rangle\) are the least levels containing \(a\) and \(h\). Then:
\begin{itemize}
\item For all \(H \in \Psi_S(\Omega_S(P, \widehat L))\) and for all nontrivial transvections \(g \notin \eunit(S^{-1}P, S^{-1}L)\) there exists \(H' \in \Psi_S(\Omega^L_S(P))\) such that \(H' \leq {^H \langle g \rangle}\).
\item For all \(H \in \Psi_S(\Omega^L_S(P))\) and for all elementary transvections \(f \in \eunit(S^{-1}P, S^{-1}\widehat L)\) there is a nontrivial transvection in \({^f H}\).
\item For all \(H \in \Psi_S(\Omega_S(P, \widehat L))\), \(f \in \eunit(S^{-1} P, S^{-1}\widehat L)\) and for all nontrivial transvection \(g \notin \eunit(S^{-1} P, S^{-1} L)\) there is \(H' \in \Psi_S(\Omega^L_S(P))\) such that \(H' \leq {^H\langle {^fg} \rangle}\).
\end{itemize}
\end{lemma}
\begin{proof}
The first two claims easily follow from lemma \ref{LevelTransvections}. In the third claim let \(f = f_n \ldots f_1\), where all \(f_i\) are elementary transvections. We will prove the claim by induction on \(n\), the case \(n = 0\) is exactly the first claim. If the claim is true for \(n - 1\), then
\begin{align*}
{^H \langle {^fg} \rangle} &\geq {^H ({^{f_n} \langle {^{H''} ({^{f_{n - 1} \ldots f_1}g})}\rangle })} \geq &&\text{ for } H'' \in \Psi_S(\Omega_S(P, \widehat L)) \text{, if } {^f H''} \leq H\\
&\geq {^H ({^{f_n} H'''})} \geq &&\text{ for some } H''' \in \Psi_S(\Omega^L_S(P))\\
&\geq {^H g'} \geq &&\text{ for a nontrivial transvection } g'\\
&\geq H^{\mathrm{IV}} &&\text{ for some } H^{\mathrm{IV}} \in \Psi_S(\Omega^L_S(P)). \qedhere
\end{align*}
\end{proof}

Let us prove the main theorem. Let \(G \leq C^*\) be a subgroup normalized by \(\eunit(P)\) and \(L\) be its level. We will prove \(G \leq \gunit(P, \lfloor L \rfloor)\) by contradiction. Suppose that it is false, then we want to find a nontrivial transvection in \(G\), i.\,e. an elementary transvection not in \(\eunit(P, L)\). In proposition \ref{LocalExtraction} (which will be proved in the next section) we will prove the assertion if \(K\) is local and \(C\) is a finite \(K\)-algebra, and the resulting nontrivial transvection can be obtained from some \(g \in G \setminus \gunit(P, \lfloor L \rfloor)\) by the rules EX\(1\) -- EX\(4\).

\begin{theorem}\label{Classification}
Suppose that \(l \geq 4\), \(C\) is a quasi-finite \(K\)-algebra, and that \(e_0 C e_0\) is generated by less than \(4l^2\) elements as a \(K\)-module.
If \(G \leq \gl(P)\) is normalized by the group \(\eunit(P)\) and has a level \(L\), then \(\eunit(P, L) \leq G \leq \gunit(P, \lfloor L \rfloor)\). Conversely, all subgroups satisfying these inequalities are normalized by \(\eunit(P)\).
\end{theorem}
\begin{proof}
Suppose that \(g \in G \setminus \gunit'(P, \lceil L \rceil)\) (recall that \(\gunit'(P, \lceil L \rceil) = \gunit(P, \lfloor L \rfloor)\) by theorem \ref{CommutantFormula}). Without loss of generality, \(\herm(K)\) is Noetherian (for example, it is finitely generated over \(\mathbb Z\)), both \(K\) and \(C\) are finite \(K\)-algebras. It is easy to see that \(\Psi_{\mathfrak m}(g) \notin \gunit'(P_{\mathfrak m}, \lceil L_{\mathfrak m} \rceil)\) for some maximal ideal \(\mathfrak m \leqt \herm(K)\) (here \(\lceil L_{\mathfrak m} \rceil = \lceil L \rceil_{\mathfrak m}\) because of Noetherian condition). Let us apply proposition \ref{LocalExtraction}, then there is a sequence of elements \(g_0 = \Psi_{\mathfrak m}(g), \ldots, g_{N + 1}\) such that \(g_{N + 1}\) is a nontrivial transvection (it can also be a product of two elementary transvections with an angle \(\frac \pi 4\) between their roots) and \(g_{N + 1}^h \in \langle [\Psi_S(g), H], H' \rangle\) for some \(h \in \eunit(S^{-1} P, S^{-1} \widehat L)\). The groups \(H \in \Psi_{\mathfrak m}(\Omega_{\mathfrak m}(P, \widehat L))\) and \(H' \in \Psi_{\mathfrak m}(\Omega_{\mathfrak m}(P, L))\) can be chosen arbitrary, if only \(H' \leq H\). We take \(H = \Psi_{\mathfrak m}(\eunit(P, \widehat L \cdot s))\) and \(H' = \Psi_{\mathfrak m}(\eunit(P, L \cdot s))\) for such \(s\) that \(\Psi_{\mathfrak m}|_{s C}\) in injective (it suffices to take such \(s\) that \(\mathrm{Ann}(s) = \Ker(\Psi_{\mathfrak m})\)). By lemma \ref{ContinuityExtraction} there is \(H'' \in \Psi_{\mathfrak m}(\Omega^L_{\mathfrak m}(P))\) such that \(H'' \leq {^H \langle g_{N + 1}^f \rangle}\). But then \(G\) itself has a nontrivial transvection (it is some expression of \(g\) modulo \(\Ker(\Psi_{\mathfrak m})\), and we can eliminate the modulus because of our choice of \(s_0\)), this is a contradiction.

The second claim follows from theorem \ref{CommutantFormula}.
\end{proof}

\section{Local extraction of transvections}

Let \(K\) be local with the maximal ideal \(\mathfrak m\), \(C\) be a finite \(K\)-algebra, \(L\) be a level, and \(g \in C^*\). We want to prove that either \(g \in \gunit(P, \lfloor L \rfloor)\), or it is possible to obtain a nontrivial transvection from \(g\) using the rules EX\(1\) -- EX\(4\).

The idea of the extraction of transvections is the following: first of all, we will do it modulo the Jacobson radical \(\mathrm J(C)\) of the algebra \(C\) note that (\(C / \mathrm J(C)\) is a semi-simple \((K / \mathfrak m)\)-algebra, because \(C\) is semi-local). After that if we have found a nontrivial transvection then it can be used to contsruct a nontrivial transvection in \(C^*\). Otherwise, we still need to extract a transvection in \(C^*\), but we can use that \([g] \in \gunit(P / \mathrm J(C), \lfloor L / \mathrm J(C) \rfloor)\) (here the right hand side means the corresponding subgroup of \((C / \mathrm J(C))^*\)). For simplicity let \(E = e_l + e_{l - 1}\).

\begin{lemma}\label{ParabolicExtraction}
Suppose that \(l \geq 4\) and \(\inv E \alpha(g) = \inv E\). Then either \(g \in \unit(P, \lceil L \rceil)\), or it is possible to extract a nontrivial transvection.
\end{lemma}
\begin{proof}
Suppose that it is impossible to extract a nontrivial transvection. Firstly suppose that \((1 - E)(\alpha(g) - 1)(1 - \inv E) = 1 - \inv E - E\). In this case \(g = \prod_{\substack{\alpha \in \Phi,\\ (\alpha, E) < 1,\\ |\alpha| < 2}} \tau_\alpha(u_\alpha)\) for some uniquely determined \(u_\alpha\). Then all factors are products of certain transvections that can be obtained from \(g\) by multiple applying of commutators and multiplications by transvections, hence they all are in \(\eunit(P, L)\).

In general case \(\inv E \alpha(g) = \inv E\) and \(\alpha(g) E = E\). If we prove that \(x = \inv E + (1 - E) g (1 - \inv E) + E \in \unit(P, \lceil L \rceil)\), then the considered case can be applied to \(gx^{-1}\). Let \(h = [g, \tau_{i, -l}(a)]\) (or \(h = [g, \tau_{i, 1 - l}(a)]\)) for \(a \in \widehat L\) and \(2 - l \leq i \leq l - 2\), \(i \neq 0\), then \((1 - E)\alpha(h)(1 - \inv E) = 1 - \inv E - E\). Therefore, \(h \in \unit(P, \lceil L \rceil)\) for all such \(i\) and \(a\), i.\,e. \((\alpha(x) - 1) e_i \in \lceil I \rceil\). From consideration of \([g, \tau_{l, l - 1}(a)]\) and \([g, \tau_{l - 1, l}(a)]\) it follows that \(e_i \alpha(g) \inv E \in I\). Returning to \(h\), it is easy to see that \(\eps \cdot h \dotminus \eps \in \lceil \Gamma \rceil\) implies \(\eps \cdot \alpha(x) e_i \dotminus \eps \cdot e_i \in \lceil \Gamma \rceil\), i.\,e. all columns \(x\) with nonzero index satisfy the required equalities. For the remaining column it can be proved using similar reasoning, if we consider \({^{\tau_i(h)} g}\) instead of \(g\) for arbitrary \(2 - l \leq i \leq l - 2\), \(i \neq 0\) and \(h \in \widehat \Gamma\) (note that \((1 - e_0) \alpha(x) e_0 \in \lceil I \rceil\) because \(e_0 \alpha(x^{-1}) (1 - e_0) \in \lceil I \rceil\)).
\end{proof}

\begin{lemma}\label{CornerExtraction}
Suppose that \(l \geq 4\) and \(\inv E \alpha(g) (\inv E + e_{-1} + e_{-2}) = \inv E\). Then either \(g \in \unit(P, \lceil L \rceil)\), or it is possible to extract a nontrivial transvection.
\end{lemma}
\begin{proof}
Assume that we cannot extract a nontrivial transvection. Let \(y = [\tau_{l - 1, -l}(a), x^{-1}]\). Clearly, \(\alpha(y) (e_{-1} + e_{-2}) = e_{-1} + e_{-2}\), hence \(y \in \unit(P, \lceil L \rceil)\) by lemma \ref{ParabolicExtraction}. Since \((1 - E) \alpha(y) \inv E = \inv E + (1 - E) \alpha(x)^{-1} (\inv a - a)\), then \(\inv E \alpha(x) (1 - \inv E) \in I\). Therefore, if we multiply \(x\) from the right by some element from \(\eunit(P, L)\), we can assume that \(\inv E \alpha(x) (1 - E) = \inv E\). Next, \(\eps \cdot \alpha(y) \inv E \in \Gamma\) and \(\inv E x E + \inv E x^{-1} E = 0\), so it easily follows that \((0, \inv E x E, 0) \in \Gamma\). Now after another multiplication of \(x\) from the right by an element from \(\eunit(P, L)\), we can assyme that \(\inv E x = \inv E\), hence lemma \ref{ParabolicExtraction} can be applied.
\end{proof}

\begin{lemma}\label{SemiparabolicExtraction}
Suppose that \(l \geq 4\), \((1 - E) g E = 0\) or \(\inv E g (1 - \inv E) = 0\), and that \(\inv E g \inv E\) and \(EgE\) are invertible in \(\inv E C \inv E\) and \(ECE\). Then either \(g \in \gunit(P, \lfloor L \rfloor)\), or it is possible to extract a nontrivial transvection.
\end{lemma}
\begin{proof}
Assume that one cannot extract a nontrivial transvection. If \((1 - E) g E = 0\), then using commutators of type \(x = [g^{-1}, \tau_{l, 1 - l}(a, b)]\) it is easy to see that \(g\) can be multiplied from the right by such an element \(f \in \eunit(P, L)\) that the product \(g' = gf\) satisfies the equalities \((1 - E) g' E = 0 = \inv E g' (1 - \inv E)\). If \(\inv E g (1 - \inv E) = 0\), then one can proceed similarly, using commutators of type \(x = [g, \tau_{l, 1 - l}(a, b)]\) and multiplying by \(f\) from the left in the definition of \(g'\). Now if we prove that \(d = \inv E g' \inv E + (1 - E - \inv E) g' (1 - E - \inv E) + Eg'E\) is in \(\gunit(P, \lfloor L \rfloor)\), then for \(g'd^{-1}\) the claim will be obvious.

We will prove that \(d \in \gunit'(P, \lceil L \rceil)\) using explicit equations from lemma \ref{GeneralLevelGroup}. Since \([d^{\pm 1}, \tau_{i, j}(a)] \in \unit(P, \lceil L \rceil)\) for \(a \in \widehat I\), \(0 \neq i \neq \pm j \neq 0\), and \(i, -j \neq l, l - 1\) (by lemma \ref{ParabolicExtraction}) and since \([d^{\pm 1}, \tau_{l, l - 1}(a)], [d^{\pm 1}, \tau_{l - 1, l}(a)] \in \unit(P, \lceil L \rceil)\) for \(a \in \widehat I\) (by lemma \ref{CornerExtraction}), then \({^{\alpha(d)^{\pm 1}} a} \equiv a \mod \lceil I \rceil\) for \(a \in \widehat I \inv E + (1 - \inv E) \lfloor \widehat I \rfloor (1 - E) + E \widehat I\). For \(a \in (1 - E) \widehat I E\) it follows from the fact that all \(b \in E \widehat I (1 - E)\) satisfy congruences \(ba \equiv \alpha(d^{-1}) b \alpha(d) a \equiv b \alpha(d) a \alpha(d^{-1}) \mod \lceil I \rceil\). The remaining equations easily follow from this.
\end{proof}

\begin{lemma}\label{AlmostDiagonalExtraction}
Suppose that \(l \geq 4\), \(e_0 C e_0\) is generated by less than \(4l^2\) elements as a \(K\)-module, \((e_{-1} \alpha(g) e_{-1}\) and \(e_{-2} \alpha(g) e_{-2}\) are invertible in \(e_{-1} \mathcal A e_{-1}\) and \(e_{-2} \mathcal A e_{-2}\). Suppose also that \(e_{-1} \alpha(g) (1 - e_{-1}), (1 - e_{-1}) \alpha(g) e_{-1}, e_{-2} \alpha(g) (1 - e_{-2}), (1 - e_{-2}) \alpha(g) e_{-2} \in \mathrm J(\mathcal A)\). Then either \(g \in \gunit(P, \lfloor L \rfloor)\), or it is possible to extract a nontrivial transvection.
\end{lemma}
\begin{proof}
Suppose that a nontrivial transvection cannot be extraced. Note that if we swap indices \(1\) and \(-1\), and also swap \(2\) and \(-2\), then the condition on \(g\) will be preserved. Let us prove that \([g, \tau_{-1, 2}(a)] \in \unit(P, \lceil L \rceil)\) for all \(a \in \widehat I\).

Consider the element \(x = {^f [g, \tau_{-1, 2}(a, b)]}\) for some \(f \in \eunit(P, \widehat L)\). We will assume that \(f = e_- f e_- + e_0 + e_+ f e_+\) and \(f \equiv 1 \mod \mathrm J(C)\). If there exists such \(f\) that \(\inv E \alpha(x) (e_{-l} + \ldots + e_0) = \inv E\), then we may apply lemma \ref{CornerExtraction}. Let us take any such \(f \in \eunit(P)\) that also \(\inv E fg (e_{-1} + e_{-2}) = 0\) (it is possible, because \(e_{-1} g e_{-1}\) and \(e_{-2} g e_{-2}\) are invertible). It is easy to see that \(\inv E x (e_{-l} + \ldots + e_0) = \inv E\), hence we only need to prove that \(\inv r = \inv E \inv{(fg)^{-1}} (e_{-1} + e_{-2}) \in {\widehat I}^0 = \{x \in C \mid (x, 0) \in \widehat I\}\). In order to do this we may assume that \(a = b\) (we will need that all such \(a\) together generate the two-sided ideal \((1 - e_0) C (1 - e_0)\)). A direct calculation shows that \(\inv E x (e_1 + e_2) = w\) and \((e_{-1} + e_{-2}) x^{-1} E = \inv w + pr\) for some \(p \in (e_{-1} + e_{-2}) C (e_{-1} + e_{-2})\), which is congruent to \(-a + \inv a\) modulo the Jacobson radical, and for \(w \in \mathrm J(C)\). %Therefore, if we prove that \(x \in \unit(P, \lceil L \rceil)\), then we may proceed by considering \(\inv E \alpha(x) (e_1 + e_2)\).

Let \(y = [x^{-1}, \tau_{l, 1 - l}(u, u)]\), then \(y(e_{-1} + e_{-2} + e_0) = e_{-1} + e_{-2} + e_0\) and \(y \in \gunit(P, \lfloor L \rfloor)\) by lemma \ref{SemiparabolicExtraction}. Is is easy to see that \((e_{-1} + e_{-2}) y^{-1} \inv E = (e_{-1} + e_{-2}) x^{-1} E (\inv u - u) \inv E\) and \(E y (e_1 + e_2) = E x^{-1} E (u - \inv u) \inv E x (e_1 + e_2)\), where \(E x^{-1} E \in E + \mathrm J(C) r\). Since \(y \in \gunit(P, \lfloor L \rfloor)\), then \(((e_{-1} + e_{-2}) y^{-1} \inv E, \inv{E y (e_1 + e_2) y^{-1} (e_1 + e_2)}) \in I\), i.\,e. \((e_{-1} + e_{-2}) y^{-1} \inv E - (e_{-1} + e_{-2}) \inv{y^{-1}} (e_{-1} + e_{-2}) \inv y \inv E \in {\widehat I}^0\). Since also \((e_1 + e_2) y^{-1} (e_1 + e_2) = e_1 + e_2 + r \mathrm J(C)\), it follows that \(r \in C \inv r \mathrm J(C) + \mathrm J(C) \inv r C + {\widehat I}_0\). We a going to prove that \([r] = 0 \in A = (1 - e_0) C (1 - r_0) / ((1 - e_0) {\widehat I}_0 (1 - e_0))\). Note that \([r] \in A [\inv r] \mathrm J(A) + \mathrm J(A) [\inv r] A\), hence \([r] \in A [r] \mathrm J(A) + \mathrm J(A) [r] A\). Applying the Nakayama lemma to \(A / A [r] \mathrm J(A)\), it follows that \([r] \in A [r] \mathrm J(A)\). Then \([r] = 0\) again by the Nakayama lemma.

Now lemma \ref{LocalLevi} implies that there are \(f_1, f_2 \in \eunit(P, L)\) such that \(e_{-1} f_1 g f_2 = e_{-1} f_1 g f_2 e_{-1} = f_1 g f_2 e_{-1}\). Therefore, \(g \in \gunit(P, \lfloor L \rfloor)\) by lemma \ref{BoundaryGenerators}.
\end{proof}

Note that we can also apply the rules EX\(1\) -- EX\(4\) to \([g]\) in \(C / \mathrm J(C)\) in order to obtain a nontrivial transvection. Here we may assume that \(S = 1\). The resulting sequence can be lifted into \(C^*\) in such a way that lifts of all auxiliary transvections are in the required open subgroups.

\begin{lemma}\label{SemisimpleExtraction}
Suppose that \(l \geq 4\), \(K\) is a field, \(C\) is a finite-dimensional semi-simple \(K\)-algebra, and \(\dim_K(e_0 C e_0) < 4l^2\). Then either \(g \in \gunit(P, \lfloor L \rfloor)\), or it is possible to extract a nontrivial transvection.
\end{lemma}
\begin{proof}
Assume that we cannot extract a nontrivial transvection. Without loss of generality, \(C\) is a simple algebra as an algebra with involution (i.\,e. either \(C = \Mat(n, D)\), where \(D\) is a division algebra with an involution, or \(C = \Mat(n, D) \times \Mat(n, D^{\mathrm{op}})\) and the involution swaps the factors). If \(C = e_0 C e_0\), then one can finish the proof by using lemma \ref{ParabolicExtraction}. We will prove that \(g \in \gunit'(P, \lceil L \rceil)\) similarly to the proof of lemma \ref{AlmostDiagonalExtraction}. First of all, we will prove that \([g, \tau_{-1, 2}(a)] \in \unit(P, \lceil L \rceil)\) for all \(a \in \widehat I\).

Consider the element \(x = {^f[g, \tau_{-1, 2}(a, b)]}\), where \(f \in \eunit(P, \widehat L)\), \(f = e_- f e_- + e_0 + e_+ f e_+\), and \(\inv E fg (e_{-1} + e_{-2}) = 0\). If there is \(f\) such that \(\inv E \alpha(fg) (e_{-1} + e_{-2}) = 0\), then we can apply lemma \ref{CornerExtraction} to \(x\). In the case when \((1 - e_0) \widehat I (1 - e_0) = (1 - e_0) \mathcal A (1 - e_0)\) this is obvious, hence we may assume that \((1 - e_0) I (1 - e_0) \leq (1 - e_0) C (1 - e_0)\), i.\,e. either \((1 - e_0) I (1 - e_0) = 0\), or \((1 - e_0) I (1 - e_0) = (1 - e_0) C (1 - e_0)\).

In the case \((1 - e_0) I (1 - e_0) = 0\) consider the element \(y = [x^{-1}, \tau_{l, 1 - l}(u, u)]\), it is in \(\gunit(P, \lfloor L \rfloor)\) by lemma \ref{SemiparabolicExtraction}. The equations of \(\gunit'(P, \lceil L \rceil)\) imply that \((1 - e_0) y (1 - e_0) = 1 - e_0\). In other words, \(E(u - \inv u)\inv E = (1 - e_0) x^{-1} E (u - \inv u) \inv E x (1 - e_0) \tau_{l, 1 - l}(u, u)^{-1}\) for all \(u \in e_l C e_{1 - l}\). Hence \((1 - e_0) x^{-1} E\) has full rank (in every simple factor of \(C\) it has the same rank as \(E\)), and therefore \(\inv E x e_+ = 0\). Lemma \ref{SemiparabolicExtraction} implies that \(x \in \gunit(P, \lfloor L \rfloor)\): indeed, if we apply the lemma to \(x^{-1}\) and recall that \(\inv E x^{-1} (1 - \inv E) = 0\), then from the beginning of the proof of the lemma it follows that \((1 - E - e_0) x^{-1} E = 0\), hence \(Ex^{-1} E\) is invertible, because \((1 - e_0) x^{-1} E\) has full rank. Similarly, \(e_- \alpha(f) (e_{-1} + e_{-2}) (a - \inv b) (e_1 + e_2) \alpha(f) e_+ = e_- \alpha(fg) (e_{-1} + e_{-2}) (a - \inv b) (e_1 + e_2) \alpha(g^{-1}) \tau_{-1, 2}(a, b)^{-1} \alpha(f^{-1}) (1 - e_0)\). The factor \(e_- \alpha(fg) (e_{-1} + e_{-2})\) has full rank, hence \((e_1 + e_2) \alpha(g^{-1}) \tau_{-1, 2}(a, b)^{-1} \alpha(f^{-1}) e_- = 0\) and \((e_1 + e_2) \alpha(g) (e_{-1} + e_{-2}) = 0\). Dividing the long equality by \(e_- \alpha(f) e_-\) from the left and by \(e_+ \alpha(f)^{-1} e_+\) from the right, we obtain \((e_{-1} + e_{-2}) (a - \inv b) (e_1 + e_2) = e_- \alpha(g) (e_{-1} + e_{-2}) (a - \inv b) (e_1 + e_2) \alpha(g)^{-1} e_+\). Therefore, \(\inv E \alpha(g) (e_{-1} + e_{-2}) = 0\).

In the last case, when \((1 - e_0) I (1 - e_0) = (1 - e_0) C (1 - e_0)\), we will consider \(x' = [f'g, \tau_{-1, 2}(a, a)]\) instead of \(x\), where \(f' \in \eunit(P, L)\), \(f' = e_- f' e_- + e_- f' e_+ + e_+ f' e_+ + e_0\), and \(\inv E f'g (e_{-1} + e_{-2}) = 0\) (clearly, it is sufficient to prove \([g, \tau_{-1, 2}(a, a)] \in \unit(P, \lceil L \rceil)\)). Moreover, we assume that \(e_{-1} f'g (e_{-1} + e_{-2}) \mathcal E \neq 0\) and \(e_{-2} f'g (e_{-1} + e_{-2}) \mathcal E \neq 0\), if \((1 - e_0) g (e_{-1} + e_{-2})\mathcal E \neq 0\), for all primitive central idempotents \(\mathcal E\) in \(C\) (there are only \(1\) or \(2\) of them). Since \(\inv E x' = \inv E\), then \(x' \in \gunit(P, \lfloor L \rfloor)\) by the beginning of the proof of lemma \ref{SemiparabolicExtraction} (applied to \(\inv{{x'}^{-1}}\)) and \(\inv E \inv{(fg)^{-1}} (e_{-1} + e_{-2}) (a - \inv a) \inv{(fg)} (e_1 + e_2) = 0\). In the case \((1 - e_0) g (e_{-1} + e_{-2}) \mathcal E \neq 0\) we have \(\inv E \inv{(fg)^{-1}} (e_{-1} + e_{-2}) \inv{\mathcal E} = 0\). In another case, i.\,e. \((1 - e_0) g (e_{-1} + e_{-2}) \mathcal E = 0\), we may take \(f\) such that \(\inv E \inv{(fg)^{-1}} (e_{-1} + e_{-2}) \inv{\mathcal E} = 0\). Therefore, in any case \(x' \in \unit(P, \lceil L \rceil)\) by lemma \ref{ParabolicExtraction}.

Now \([g^{\pm 1}, \tau_\alpha(a)] \in \unit(P, \lceil L \rceil)\) for all short roots \(\alpha\). Hence there is \(h \in \eunit(P, \widehat L)\) such that \(e_{-l} \alpha({^hg}) e_{-l}\) in invertible in \(e_{-l} \mathcal A e_{-l}\) (using the proof of the proposition \ref{LocalSplitting}). But \([({^hg})^{\pm 1}, \tau_\alpha(a)]\) are also in \(\unit(P, \lceil L \rceil)\) for all short roots \(\alpha\), hence by lemma \ref{LocalLevi} there are \(f_1, f_2 \in \eunit(P, L)\) such that \(e_{-l} \alpha(f_1 \, {^hg} \, f_2) = e_{-l} \alpha(f_1 \, {^hg} \, f_2) e_{-l} = \alpha(f_1 \, {^hg} \, f_2) e_{-l}\). The rest of the proof is obvious by lemma \ref{BoundaryGenerators}.
\end{proof}

\begin{prop}\label{LocalExtraction}
Suppose that \(l \geq 4\), \(K\) is a local ring, \(C\) is a finite \(K\)-algebra, \(e_0 C e_0\) is generated by less than \(4l^2\) elements as a \(K\)-module, and \(g \in C^*\). Then either \(g \in \gunit(P, \lfloor L \rfloor)\), or it is possible to extract a nontrivial transvection.
\end{prop}
\begin{proof}
By lemma \ref{SemisimpleExtraction} either we can extract a nontrivial transvection modulo \(\mathrm J(C)\) (in this case we can extract a nontrivial transvection in \(C\) by lemma \ref{AlmostDiagonalExtraction}), or \([g] \in \gunit(P / \mathrm J(C), \lfloor L / \mathrm J(C) \rfloor)\). In the second case double application of proposition \ref{LocalSplitting} to \(g\) gives us an element \(g' = f \, {^{h\!}g}\) for \(f \in \eunit(P, L)\) and \(h \in \eunit(P, \widehat L)\) such that lemma \ref{AlmostDiagonalExtraction} can be applied to \(g'\).
\end{proof}

\section{Application to classical groups}

Let us begin from the case of even unitary groups. If \(P = \bigoplus_{i = 1}^n \herm(P_i)\) is a quadratic module over a ring \(R\) with a pseudo-involution, then by paper \cite{OddPetrov} we can assume that \(A_R = R / \Ker(\phi)\). The set \(\Lambda = \Ker(\phi) \leq R\) is a form parameter in Bak's sense (it is an additive subgroup closed under the action of \(R^\bullet\) and it is contained between \(\Lambda_{\mathrm{min}} = \{r - \inv r \lambda \mid r \in R\}\) and \(\Lambda_{\mathrm{max}} = \{r \in R \mid r + \inv r \lambda = 0\}\)). The map \(\tr\) is given by the formula \(\tr(r + \Lambda) = r + \inv r \lambda\). We also assume that all \(P_i\) are fully projective over \(R\) (for example, \(P_i = R\)) and \(l \geq 4\).

The ring \(C = E = \End_R(P_R)\) is Morita equivalent to the ring \(R\), hence the lattices of ideals of \(C\) and \(R\) are isomorphic. Since quasi-finiteness is Morita invariant, the assumption of the main theorem \ref{Classification} will be satisfied if \(R\) is quasi-finite over its center (and the center can be taken as \(K\)). The group \(\mathcal H\) in this case equals \(C\) with the addition operations and \(\Lambda\) from lemma \ref{Lambda} is Morita equivalent to the form parameter, i.\,e. \(\Lambda_i = \{x \in \End_R(P_i) \mid B(p, xp) \in \Ker(\phi) \text{ for } p \in P_i\}\).

In the even case levels and augmented levels are the same, they are just pairs \(L = (I, \Gamma)\), where \(I = \inv I \leqt \mathcal A\) and \(\Gamma \leq \mathcal H = C\) are such that \(I^2 \subseteq I\), \(\tr(\Gamma) \leq I\) and \(\Gamma \cdot I \dotplus \Gamma \cdot C \dotplus \Lambda \cdot I \dotplus \phi(I) \leq \Gamma\). The main theorem claims that if \(l \geq 4\) and \(R\) is quasi-finite, then all subgroups \(G \leq \gl(P)\) normalized by \(\eunit(P)\) are contained between \(\eunit(P, L)\) and \(\gunit(P, L)\) for unique level \(L\). Here \(\gunit(P, L) = \{g \in \gl(P) \mid [\{g, g^{-1}\}, \eunit(P, \widehat L)] \subseteq \eunit(P, L)\}\).

If we are interested only in subgroups of \(\unit(P)\) normalized by \(\eunit(P)\), then these subgroups can be described in the same way if we take \(\gunit(P, L) \cap \unit(P)\) instead of \(\gunit(P, L)\) and if we consider only the levels contained in the level \(\eunit(P)\). In other words, these subgroups are described by pairs \(L = (I, \Gamma)\), where \(I = \inv I \leqt C\) and \(\Gamma \leq \Lambda\) are such that \(\tr(\Gamma) \leq I\) and \(\Gamma \cdot C \dotplus \Lambda \cdot I \dotplus \phi(I) \leq \Gamma\). Under Morita equivalence these pairs correspond to form ideals of the ring \(R\), hence we obtain the main result of \cite{UnitaryII}, though in this paper the theorem was proved under weaker assumption \(l \geq 3\). In particular, we have the sandwich classification theorem for subgroups of \(\mathrm O(2l, K)\) and \(\mathrm{Sp}(2l, K)\) normalized by corresponding elementary subgroups, where \(K\) is a commutative ring.

Now let us classify overgroups of \(\eunit(P)\) in \(\gl(P)\). These groups are normalized by \(\eunit(P)\), hence for \(l \geq 4\) and quasi-finite \(R\) the main theorem can be applied. The groups are described through levels containing the level \(\eunit(P)\), i.\,e. pairs \(L = (I, \Gamma)\), where \(I = C \oplus J \times 0 \subseteq \mathcal A = C \times C\) and \(\Gamma \leq \mathcal H = C\) are such that \(J = \inv J \leqt C\), \(\Lambda \dotplus \Gamma \cdot C \dotplus J \leq \Gamma \leq \{x \in C \mid x + \inv x \in J\}\). Under Morita equivalence these pairs \((J, \Gamma)\) correspond to the levels from paper \cite{Petrov}, hence we obtain the main result from this paper (though in the paper there was another assumption instead of quasi-finiteness of \(R\), but it also was only for \(l \geq 4\)). In particular, in this way we have the sandwich classification of overgroups of \(\mathrm O(2l, K)\) and \(\mathrm{Sp}(2l, K)\) in the general linear group over a commutative ring \(K\), i\,e. the results from papers \cite{OrthogonalEven} and \cite{Symplectic}.

In the case of odd orthogonal group \(\mathrm O(2l + 1, K)\) the main theorem cannot be applied directly, because we have assumed nondegeneracy of the hermitian form and for odd orthogonal group the form may be degenerate if two is not invertible in \(K\). Therefore, in order to deal with it properly, we should generalize the classification theorem for the case of \(P_0\) with degenerate form, where one cannot freely use proposition \ref{QuadraticMorita}.

Fortunately, we can circumvent this in the classification of overgroups of \(\mathrm{EO}(2l + 1, K)\) in \(\gl(2l + 1, K)\). Take a nondegenerate symmetric bilinear form \(B(x, y) = xy\) on \(P_0 = K\) and take the corresponding quadratic form from the odd form parameter \(\mathcal L = 0\). If we take \(P_i = K\) for \(1 \leq i \leq l\), then we will obtain the quadratic space \(P = K^{2l + 1}\), its unitary group equals \(\unit(P) = \mathrm O(2l, K)\). Indeed, if \(g \in \gl(P)\) preserves the quadratic form, then this element trivially acts on \(P_0\). If \(g\) also preserves the symmetric bilinear form, then it reduces to an orthogonal operator on \(P_0^\perp = \bigoplus_{i > 1} \herm(P_i)\). In this case our theorem gives among other things the sandwich classification of overgroups of \(\mathrm{EO}(2l, K)\) in \(\gl(2l + 1, K)\). The overgroups of \(\mathrm{EO}(2l + 1, K)\) will be obtained from the theorem, if we calculate the level of \(\mathrm{EO}(2l + 1, K)\).

For the module \(P = K^{2l + 1}\) we have \(C = E = \Mat(2l + 1, K)\) with the involution \(\inv{(a_{i, j})_{i, j = -l}^l} = (a_{-j, -i})_{i, j = -l}^l\) (it is the reflection across the antidiagonal), \(\mathcal A = E \times E\), and \(\mathcal H = e_0 E \times E \times Ee_0 = {^{2k + 1}\!K} \times \Mat(2l + 1, K) \times K^{2l + 1}\), where \(K^{2l + 1}\) is the column space and \({^{2k + 1}\!K}\) is the row space. Also \(\Lambda_i = 0\), i.\,e. \(\Lambda = \phi(E) \cdot (1 - e_0)\). An augmented level is a pair \(L = (I, \Gamma)\), where \(I = \inv I \leq \mathcal A\) is a \(K\)-submodule and \(\Gamma \leq \mathcal H\) are such that \(I^2 \leq I\), \(I (1 - e_0) E (1 - e_0) \leq I\), \(\pi(\Gamma) \leq I\), \(\tr(\Gamma) + \inv{\pi(\Gamma)} \pi(\Gamma) \leq I\), \(\Gamma \cdot (I + K + (1 - e_0) E (1 - e_0)) \dotplus \phi(I) \leq \Gamma\), and \(e_0 I (1 - e_0) = \pi(\Gamma \cdot (1 - e_0))\). The group \(\mathrm{EO}(2l, K)\) has the level \(L_0 = ((1 - e_0) E (1 - e_0), \Lambda)\) and the group \(\mathrm{EO}(2l + 1, K)\) has the level \(L_1 = (I_1, \Gamma_1)\), where \((1 - e_0) I_1 (1 - e_0) = (1 - e_0) E (1 - e_0)\) and \(\Gamma_1 \cdot (1 - e_0) = \{(x, y, z) \in \mathcal H \cdot (1 - e_0) \mid z = 2 \inv x, y = \inv x x\} \dotplus \phi(E)\).

This can be used in the classification of overgroups of \(\mathrm{EO}(2l + 1, K)\) for \(l \geq 4\). Let \(L = (I, \Gamma)\) be an arbitraty level containing \(L_1\). It is uniquely detreminted by the groups \((1 - e_0) I (1 - e_0) = (1 - e_0) E (1 - e_0) \oplus 0 \times J\), \(e_0 I (1 - e_0) = \{(x, -2x) \mid x \in {^{2l}K}\} \oplus 0 \times \mathfrak b^{2l}\), and \(\Gamma \cdot e_1\), where \(J \leqt (1 - e_0) E (1 - e_0)\) as a two-sided ideal is Morita-equivalent to the ideal \(\mathfrak a \leqt K\), \(\mathfrak b \leqt K\) is also an ideal, and the group \(\Gamma \cdot e_1 \leq \mathcal H \cdot e_1 \cong K \times K \times K\) is of type \(\{(x, x^2, 2x)\} \dotplus W\), where \(0 \times \mathfrak a \times 0 \leq W \leq \{(0, y, z) \in K \times K \times K \mid 2y \in \mathfrak a, z \in \mathfrak b\}\) is uniquely determined. Moreover, \(2\mathfrak a \leq \mathfrak b \leq \mathfrak a\), for every \(z \in \mathfrak b\) there is \(y\) such that \((0, y, z) \in W\), and for all \((0, y, z) \in W\), \(a \in \mathfrak a\), and \(k \in K\) the elements \((0, k^2 y, kz)\) and \((0, 0, az)\) are also in \(W\). Conversely, if \(\mathfrak a, \mathfrak b\) and \(W\) satisfy these conditions, then they can be obtained from some level containing \(L_1\). In the case \(2 \in K^*\) we have \(\mathfrak a = \mathfrak b\) and \(W = 0 \times \mathfrak a \times \mathfrak a\), hence we obtain the result from paper \cite{OrthogonalOdd} (where this was proved under weaker assumption \(l \geq 3\)).

Finally, we can classify the subgroups of \(\mathrm O(2l + 1, K)\) that contain \(\mathrm{EO}(2l, K)\) for \(l \geq 4\). Let \(L = (I, \Gamma)\) be a level between \(L_0\) and \(L_1\). It is uniquely determined by the groups \(e_0 I (1 - e_0) \leq \{(x, -2\inv x) \mid x \in {^{2l}\!K}\}\) and \(\Gamma \cdot e_1\). Here \(e_0 I (1 - e_0) = \{(x, -2x) \mid x \in {^{2l}\!\mathfrak a}\}\) for unique ideal \(\mathfrak a \leqt K\) and \(\Gamma \cdot e_1 \leq \mathcal H \cdot e_1 \cong K \times K \times K\) is of type \(\{(x, x^2, 2x) \mid x \in \mathfrak b\}\) for unique ideal \(\mathfrak b \leqt K\). The condition \(\pi(\Gamma \cdot e_1) = e_0 I e_1\) means that \(\mathfrak a = \mathfrak b\). Conversely, every ideal \(\mathfrak a\) can be obtained from unique level between \(L_0\) and \(L_1\). Therefore, the overgroups of \(\mathrm{EO}(2l, K)\) in \(\mathrm O(2l + 1, K)\) are classified by the ideals of \(K\).

\bibliographystyle{plain}  
\bibliography{references}  %%% Remove comment to use the external .bib file (using bibtex).

\begin{thebibliography}{10}

\bibitem{Artin}
E.~Artin.
\newblock {\em Geometric algebra}.
\newblock Inters. Publ., New York, 1957.

\bibitem{BakThesis}
A.~Bak.
\newblock {\em The stable structure of quadratic modules}.
\newblock Thesis Columbia Univ., 1969.

\bibitem{Bak}
A.~Bak.
\newblock {\em ${K}$-Theory of forms}.
\newblock Ann. of Math. Stud. 98. Princeton Univ. Press, Princeton, 1982.

\bibitem{NonabelianBak}
A.~Bak.
\newblock Nonabelian ${K}$-theory: the nilpotent class of ${K}_1$ and general
  stability.
\newblock {\em K-Theory}, 4:363--397, 1991.

\bibitem{BakHazratVav}
A.~Bak, R.~Hazrat, and N.~Vavilov.
\newblock Localisation-completion strikes again: relative ${K}_1$ is nilpotent
  by abelian.
\newblock {\em J. Pure Appl. Algebra}, 213:1075--1085, 2009.

\bibitem{OddPreusser}
A.~Bak and R.~Preusser.
\newblock The {E}-normal structure of odd-dimensional unitary groups.
\newblock {\em J. Pure Appl. Algebra}, 222(9):2823--2880, 2018.

\bibitem{BakVav}
A.~Bak and N.~Vavilov.
\newblock Normality for elementary subgroup functors.
\newblock {\em Math. Proc. Cambridge Philos. Soc.}, 118(1):35--47, 1995.

\bibitem{UnitaryI}
A.~Bak and N.~Vavilov.
\newblock Structure of hyperbolic unitary groups {I}: elementary subgroups.
\newblock {\em Algebra Colloq.}, 7(2):159--196, 2000.

\bibitem{Bass2}
H.~Bass.
\newblock ${K}$-theory and stable algebra.
\newblock {\em Publ. Math. Inst. Hautes Etudes Sci.}, 22:5--60, 1964.

\bibitem{Bass1}
H.~Bass.
\newblock Unitary algebraic ${K}$-theory.
\newblock {\em Lecture Notes Math.}, 343:57--265, 1973.

\bibitem{Dieudonne2}
J.~Dieudonn\'e.
\newblock On the automorphism of the classical groups.
\newblock {\em Mem. Amer. Math. Soc.}, 2:1--122, 1951.

\bibitem{Dieudonne3}
J.~Dieudonn\'e.
\newblock {\em La g\'eometrie des groupes classiques. 3\'eme ed.}
\newblock Springer-Verlag, Berlin, 1971.

\bibitem{Dieudonne1}
J.~Dieudonn\'e.
\newblock {\em Sur les groupes classiques. 3\'eme ed.}
\newblock Hermann. Paris, 1973.

\bibitem{Golub}
I.~Golubchik.
\newblock On the normal subgroups of orthogonal group over an associative ring
  with involution.
\newblock {\em Uspehi Mat. Nauk.}, 30(6):165, 1975.

\bibitem{HahnO}
A.~Hahn and O.T. O'Meara.
\newblock {\em The classical groups and ${K}$-theory}.
\newblock Springer, Berlin, 1989.

\bibitem{Hazrat}
R.~Hazrat.
\newblock Dimension theory and nonstable ${K}_1$ of quadratic modules.
\newblock {\em K-Theory}, 27:293--328, 2002.

\bibitem{Bak65}
R.~Hazrat and N.~Vavilov.
\newblock Bak's work on the ${K}$-theory of rings.
\newblock {\em J. K-Theory}, 4:1--65, 2009.

\bibitem{Hong1}
Y.~Hong.
\newblock Overgroups of symplectic groups in linear group over commutative
  rings.
\newblock {\em J. Algebra}, 282:23--32, 2004.

\bibitem{Hong2}
Y.~Hong.
\newblock Overgroups of classical groups over commutative ring in linear group.
\newblock {\em J. Algebra}, 304:1004--1013, 2006.

\bibitem{Knus}
M.-A. Knus.
\newblock {\em Quadratic and hermitian forms over rings}.
\newblock Springer-Verlag, 1991.

\bibitem{Li1}
Fuan Li.
\newblock The structure of symplectic group over arbitrary commutative rings.
\newblock {\em Acta Math. Sinica, New Series}, 3(3):247--255, 1987.

\bibitem{Li2}
Fuan Li.
\newblock The structure of orthogonal groups over arbitrary commutative rings.
\newblock {\em Chinese Ann. Math.}, 10B(3):341--350, 1989.

\bibitem{Petrov}
V.A. Petrov.
\newblock Overgroups of unitary groups.
\newblock {\em K-Theory}, 29:147--174, 2003.

\bibitem{OddPetrov}
V.A. Petrov.
\newblock Odd unitary groups.
\newblock {\em J. Math. Sci.}, 130(3):4752--4766, 2005.

\bibitem{EvenCommutative}
R.~Preusser.
\newblock Sandwich classification for $\mathrm{{GL}}_n({R})$,
  $\mathrm{{O}}_{2n}({R})$ and $\mathrm{{U}}_{2n}({R}, {\Lambda})$ revisited.
\newblock {\em J. Group Theory}, 21:21--44, 2017.

\bibitem{UnitaryII}
R.~Preusser.
\newblock Structure of hyperbolic unitary groups {II}: classification of
  {E}-normal subgroups.
\newblock {\em Algebra Colloq.}, 24(2):195--232, 2017.

\bibitem{OddCommutative}
R.~Preusser.
\newblock Sandwich classification for $\mathrm{{O}}_{2n + 1}({R})$ and
  $\mathrm{{U}}_{2n + 1}({R}, {\Delta})$ revisited.
\newblock {\em J. Group Theory}, 21:539--571, 2018.

\bibitem{Sus}
A.~Suslin.
\newblock On the structure of the general linear group over polynomial rings.
\newblock {\em Soviet Math. Izv.}, 41(2):503--516, 1977.

\bibitem{Tang1}
Guoping Tang.
\newblock Hermitian groups and ${K}$-theory.
\newblock {\em K-Theory}, 13(3):209--267, 1998.

\bibitem{Tits}
J.~Tits.
\newblock Formes quadratiques, groupes orthogonaux et alg\'ebres de {C}lifford.
\newblock {\em Invent. Math.}, 5:19--41, 1968.

\bibitem{Vaser1}
L.~Vaserstein.
\newblock Stabilization of unitary and orthogonal groups over a ring with
  involution.
\newblock {\em Math. USSR Sbornik}, 10:307--326, 1970.

\bibitem{Vaser2}
L.~Vaserstein.
\newblock Stabilization for classical groups over rings.
\newblock {\em Math. USSR-Sb.}, 22(2):271--303, 1974.

\bibitem{Vaser3}
L.~Vaserstein.
\newblock Normal subgroups of orthogonal groups over commutative rings.
\newblock {\em Amer. J. Math.}, 110(5):955--973, 1988.

\bibitem{Vaser4}
L.~Vaserstein.
\newblock Normal subgroups of symplectic groups over rings.
\newblock {\em K-Theory}, 2(5):647--673, 1989.

\bibitem{VaserYou}
L.~Vaserstein and Hong You.
\newblock Normal subgroups of classical groups over rings.
\newblock {\em J. Pure Appl. Algebra}, 105:93--106, 1995.

\bibitem{Vav1}
N.~Vavilov.
\newblock Intermediate subgroups in {C}hevalley groups.
\newblock In {\em Proc. Conf. Groups of Lie Type and their Geometries (Como --
  1993)}, pages 233--280. Cambridge Univ. Press., 1995.

\bibitem{OrthogonalEven}
N.A. Vavilov and V.A. Petrov.
\newblock On the overgroups of $\mathrm{{EO}}(2l, {R})$.
\newblock {\em J. Math. Sci.}, 116(1):2917--2925, 2003.

\bibitem{Symplectic}
N.A. Vavilov and V.A. Petrov.
\newblock Overgroups of elementary symplectic groups.
\newblock {\em St. Petersburg Math. J.}, 15:515--543, 2004.

\bibitem{OrthogonalOdd}
N.A. Vavilov and V.A. Petrov.
\newblock Overgroups of $\mathrm{{EO}}(n, {R})$.
\newblock {\em St. Petersburg Math. J.}, 19:167--195, 2008.

\bibitem{VavStep}
N.A. Vavilov and A.V. Stepanov.
\newblock Overgroups of semi-simple groups.
\newblock {\em Vestnik Samara State Univ.}, 3(62):51--95, 2008.
\newblock In russian.

\bibitem{Wall}
Ch.T.C. Wall.
\newblock On the axiomatic foundation of the theory of hermitian forms.
\newblock {\em Proc. Cambridge. Phil. Soc.}, 67:243--250, 1970.

\bibitem{Weil}
A.~Weil.
\newblock Classical groups and algebras with involution.
\newblock {\em J. Indian Math. Soc.}, 24:589--623, 1961.

\bibitem{Wilson}
J.~Wilson.
\newblock The normal ans subnormal structure of general linear groups.
\newblock {\em Proc. Camb. Phil. Soc.}, 71:163--177, 1972.

\end{thebibliography}

\iffalse

\fi

\end{document}